\documentclass[11pt]{amsart}
\usepackage{amssymb,amsfonts,amsthm,amscd,stmaryrd,dsfont,esint,upgreek}
\usepackage[numbers]{natbib}
\usepackage[mathcal]{euscript}
\usepackage[usenames]{color}
\usepackage[top=1.25in,bottom=1.5in, left=1.25in,right=1.25in]{geometry}
\theoremstyle{plain}
\newtheorem{thm}{Theorem}

\newtheorem{lem}[thm]{Lemma}
\newtheorem{prop}[thm]{Proposition}

\theoremstyle{definition}
\newtheorem{definition}[thm]{Definition}
\newtheorem{remark}[thm]{Remark}

\usepackage{color} 
\usepackage[usenames,dvipsnames]{xcolor}

\definecolor{orange}{RGB}{255,127,0}

\numberwithin{thm}{section}
\numberwithin{equation}{section}

\newcommand{\R}{\mathbb{R}}

\newcommand{\N}{\mathbb{N}}
\newcommand{\Z}{\mathbb{Z}}
\newcommand{\Sy}{\mathbb{S}^\d}

\renewcommand{\d}{d}
\newcommand{\m}{m}
\newcommand{\db}{\mathrm{d}}

\newcommand{\Rd}{\mathbb R^\d}

\newcommand{\ep}{\varepsilon}

\renewcommand{\tilde}{\widetilde}

\DeclareMathOperator*{\osc}{osc}

\DeclareMathOperator{\tr}{tr}

\DeclareMathOperator{\USC}{USC}
\DeclareMathOperator{\LSC}{LSC}

\DeclareMathOperator{\dist}{dist}

\def\XXint#1#2#3{{\setbox0=\hbox{$#1{#2#3}{\int}$}
     \vcenter{\hbox{$#2#3$}}\kern-.5\wd0}}

\begin{document}

\title[Viscous Hamilton-Jacobi equations]{Viscosity solutions of general viscous Hamilton-Jacobi equations}

\author[S. N. Armstrong]{Scott N. Armstrong}
\address{CEREMADE (UMR CNRS 7534), Universit\'e Paris-Dauphine, Paris, France}
\email{armstrong@ceremade.dauphine.fr}

\author[H. V. Tran]{Hung V. Tran}
\address{Department of Mathematics\\
The University of Chicago\\ 5734 S. University Avenue Chicago, Illinois 60637, USA}
\email{hung@math.uchicago.edu}

\date{\today}

\begin{abstract}
We present comparison principles, Lipschitz estimates and study state constraints problems for degenerate, second-order Hamilton-Jacobi equations. 
\end{abstract}

\maketitle

\section{Introduction} \label{I}

\subsection{Motivation and summary of results}
In this paper we present the basic PDE theory of viscosity solutions of the second-order ``viscous" Hamilton-Jacobi equation
\begin{equation}\label{e.vhj}
u_t -  \tr(A(x) D^2u) + H(Du,x) = 0  \ \quad \mbox{in} \ \Rd\times (0,\infty)
\end{equation}
as well as the time-independent analogues such as
\begin{equation}\label{e.vhj2}
u -  \tr(A(x) D^2u) + H(Du,x) = 0  \ \quad \mbox{in} \ \Rd.
\end{equation}
The core assumptions are that the Hamiltonian $H:\Rd \times \Rd \to\R$ is convex and superlinear (``coercive") in its first variable and the diffusion matrix $A:\Rd \to \Sy$ is degenerate elliptic (i.e., nonnegative definite and uniformly bounded). A typical example is
\begin{equation*}\label{}
u_t - a(x)^2 \Delta u + b(x)|Du|^2 = 0
\end{equation*}
with Lipschitz coefficients $a:\Rd\to [0,1]$ and $b:\Rd \to [1,2]$. Particular members of this family of partial differential equations arise in deterministic and stochastic optimal control, dynamical systems, and the study of large-scale behavior of diffusions in heterogeneous environments.

The basic theory of~\eqref{e.vhj} and~\eqref{e.vhj2} is still incomplete in many important respects, despite having received much attention over the last 25 years. This is primarily due to the difficulties imposed by the combination of a heterogeneous (and in general, degenerate) diffusion and a heterogeneous Hamiltonian, which is precisely the sort of situation encountered in stochastic homogenization. Indeed, the motivation for writing this paper originated in the theory of stochastic homogenization and our discovery that some results needed for the latter were at best ``folklore" (known to some experts but not appearing in the literature) and many others were simply open problems. 

The first contribution of this paper is a general comparison principle for these equations: previously, comparison principles were known for~\eqref{e.vhj} or~\eqref{e.vhj2} in special cases (e.g., if $A\equiv 0$, or if the dependence of $H$ on $Du$ and $x$ is decoupled in a weak sense, or $H$ is superquadratic in $Du$, or if $A$ is uniformly positive definite, among other special circumstances, see below for more discussion). Note that~\eqref{e.vhj} is not under the classical comparison regime of~\cite{CIL} since hypothesis (3.14) of that paper does not hold in general, as in the case that the Hamiltonian has a term like $b(x)|Du|^2$. The comparison principle is given in Section~\ref{MCP}.

In Section~\ref{S.A}, we present interior Lipschitz estimates for continuous solutions of~\eqref{e.vhj} and~\eqref{e.vhj2}. These estimates do not use the convexity of $H$ and so they hold for a general class of nonconvex equations. Such estimates are well-known to experts, at least in special cases.  The novelty here is that we give a new argument which is both robust enough to handle general equations and gives an~\emph{explicit} Lipschitz constant in terms of the structural hypotheses (and is essentially optimal). Such explicit estimates are important for stochastic homogenization (c.f.~\cite{AT}).

Finally, we present new results for solutions of~\eqref{e.vhj2} in bounded domains subject to \emph{state constrained boundary conditions}. This is a much-studied topic that originated in the work of Lasry and Lions~\cite{LL} and has important applications in stochastic optimal control. Due to difficulties arising in handling these special boundary conditions, there are few results for equation with anisotropic diffusions (it is usually assumed that the diffusion term vanishes or else is the Laplacian) unless $H$ grows superlinearly (which is essentially the same as $A\equiv 0$). Handling a constant diffusion matrix is easier, because in this case it is possible to obtain a precise blow-up rate for solutions near the boundary of the domain-- which then allows for comparison arguments. Here we introduce a new idea which allows us to obtain partial comparison for the state constraints problem for a general class of equations (with a general degenerate, anisotropic $A$). In particular, we show that there exists a unique maximal solution which is continuous (and hence Lipschitz). We develop an analogous theory for the \emph{metric problem}, motivated by problems in stochastic homogenization. These results can be found in Sections~\ref{BVSC} and~\ref{s.mp}.

\subsection{Hypotheses on the coefficients}
\label{ss.hypos}

The following conditions on the coefficients are assumed to be in force throughout the paper. We fix parameters $\m> 1$, $n\in \N$, $\Lambda_1\geq 1$ and $\Lambda_2 \geq 0$. We require the diffusion matrix $A:\Rd\to \Sy$ to have a Lipschitz square root, that is, there exists $\sigma :\Rd \to \R^{\d\times n}$ such that 
\begin{equation*} \label{}
A = \frac12 \sigma^t \sigma
\end{equation*}
and, for every $x,y\in B_2$, we have
\begin{equation}\label{e.sigbnd}
\left| \sigma(x)\right| \leq \Lambda_2
\end{equation}
and
\begin{equation}\label{e.siglip}
\left| \sigma(x) - \sigma(y) \right| \leq \Lambda_2 |x-y|.
\end{equation}
As for the Hamiltonian $H:\Rd \times \Rd \to \R$, we require the following: for every $x\in \Rd$,
\begin{equation}\label{e.Hconvex}
p\mapsto H(p,x) \quad \mbox{is convex.}
\end{equation}
For every $R>0$, there exist constants $0< a_R \leq 1$ and $M_R \geq 1$ such that, for every $p,q\in \Rd$ and $x,y\in B_R$,
\begin{equation}\label{e.Hsubq}
 a_R |p|^{\m} - M_R \leq H(p,x) \leq \Lambda_1\big( |p|^\m+1 \big),
\end{equation}
\begin{equation}\label{e.HsubqLip}
\left| H(p,x) - H(p,y) \right| \leq \big( \Lambda_1 |p|^\m + M_R \big) |x-y|, 
\end{equation}
and
\begin{equation}\label{e.HsubqDp}
\left| H(p,x) - H(q,x) \right| \leq \Lambda_1 \big( |p| + |q| + 1 \big)^{\m-1} |p-q|. 
\end{equation}
If the constants $a_R$ and $M_R$ can be chosen to be independent of $R$, then we say that the Hamiltonian is \emph{uniformly coercive}. Otherwise, we say that $H$ is weakly coercive. 

We emphasize that the diffusion matrix $A$ can be degenerate in general, 
and \eqref{e.HsubqLip} holds for some given $m>1$, which is more general than
hypothesis (3.14) in \cite{CIL}.

\subsection{Viscosity solution preliminaries}\textbf{}
Unless otherwise indicated, each of the differential inequalities in this paper are to be interpreted in the \emph{viscosity} sense, which is the usual notion of weak solution for Hamilton-Jacobi equations. The reader may consult~\cite{CIL}.

For technical reasons, it is convenient to work with the well-known extension (introduced in~\cite{BP}) of the definition of viscosity solutions to possibly discontinuous, locally bounded functions. We recall the definitions for the readers' convenience. If $u: V \to \R$ is locally bounded, then the \emph{upper semicontinuous envelope} $u^\ast$ of $u$ in $V$ is defined for $x\in \overline V$ by
\begin{equation*} \label{}
u^\ast(x):= \inf\left\{ w(x) \, : \, w\in \USC(\overline V) \ \mbox{and} \ w\geq u \right\} = \inf_{\delta > 0}  \sup_{B_\delta(x)} u.
\end{equation*}
Here $\USC(V)$ denotes the set of upper semicontinuous functions on $V$, taking values in $\R \cup\{ +\infty\}$, and we note that $u^\ast$ belongs to $\USC(\overline V)$ since $\USC(\overline V)$ is closed under taking infimums. We likewise define the \emph{lower semicontinuous envelope} $u_\ast\in \LSC(V)$ of $u$ in $V$ by $u_\ast:= -(-u)^\ast$.

\begin{definition}[Viscosity solution]\label{defvs}
We say that a function $u$ is a \emph{viscosity subsolution (solution)} of the differential equation (inequality) 
\begin{equation*} \label{}
u_t - \tr\left( A(x) D^2u \right) + H\!\left( Du,x \right) = \ \mbox{($\leq$)} \ 0 \quad \mbox{in} \ V \subseteq \Rd \times \R_+
\end{equation*}
if $u:V \to \R$ is locally bounded from above and, for every $(y,s) \in V$ and smooth function $\varphi$ which is defined in a neighborhood of $(y,s)$ such that
\begin{equation} \label{def.vsup}
(x,t)\mapsto \left( u^\ast - \varphi\right) (x,t) \quad \mbox{has a local maximum at} \ (x,t) = (y,s),
\end{equation}
then we have
\begin{equation*} \label{}
\varphi_t(y,s) - \tr\left( A(y) D^2\varphi(y,s) \right) + H\!\left( D\varphi(y,s),y \right) \leq  0.
\end{equation*}
Likewise, $u:V\to \R$ is a \emph{viscosity supersolution (solution)} of the differential equation (inequality) 
\begin{equation} \label{def.vsdn}
u_t - \tr\left( A(x) D^2u \right) + H\!\left( Du,x \right) = \ \mbox{($\geq$)} \ 0 \quad \mbox{in} \ V 
\end{equation}
if $u$ is locally bounded from below and, for every $(y,s) \in V$ and smooth function $\varphi$ which is defined in a neighborhood of $(y,s)$ such that
\begin{equation*} \label{}
(x,t)\mapsto \left( u_\ast - \varphi\right) (x,t) \quad \mbox{has a local minimum at} \  (x,t)=(y,s),
\end{equation*}
then we have
\begin{equation*} \label{}
\varphi_t(y,s) - \tr\left( A(y) D^2\varphi(y,s) \right) + H\!\left( D\varphi(y,s),y \right) \geq  0.
\end{equation*}
We say that $u:V\to \R$ is a \emph{viscosity solution} of 
\begin{equation*} \label{}
u_t - \tr\left( A(x) D^2u \right) + H\!\left( Du,x \right) = 0 \quad \mbox{in} \ V 
\end{equation*}
if $u$ is locally bounded and both a viscosity subsolution and supersolution. 
\end{definition}

The definition of viscosity solution for other equations considered here (e.g. equations with no time dependence) is identical. We remark that there is a well-known equivalence between the definition above and the alternative (weaker) definition in which the local maxima/minima in~\eqref{def.vsup}--\eqref{def.vsdn} are strict.

\section{Comparison principles}
\label{MCP}

\subsection{Comparison principles for stationary problems}

We present comparison results for time-independent problems. The arguments combine several ingredients, none of which are new. Besides the classical comparison argument for viscosity solutions~\cite{CIL}, we need an idea based on the convexity of $H$ that goes back at least to Barles and Perthame~\cite{BP} and appears in a form closer to our argument in Da Lio and Ley~\cite{DaL1}. See also Barles and Da Lio~\cite{BF} as well as \cite{DaL2,KL,AS1,BBBL}, wherein special cases of the results presented in this section can also be found. See also Kobylanski~\cite{K} for some related probabilistic results.

\begin{thm}\label{a.comp.mono}
Let $\delta > 0$, $U\subseteq \Rd$ be open and assume that~$u,-v\in \USC(\overline U)$ satisfy
\begin{equation}\label{e.comp.deq}
\delta u - \tr\left( A(x) D^2u\right) + H(Du,x) \leq 0 \leq \delta v - \tr\left( A(x) D^2v\right) + H(Dv,x) \quad \mbox{in} \ U
\end{equation}
as well as
\begin{equation}\label{e.comp.growth}
u\leq v\quad \mbox{on} \ \partial U, \quad  \mbox{and} \quad \limsup_{x \in U, \ |x| \to \infty} \frac{u(x)}{1+|x|} \leq 0 \leq \liminf_{x \in U, \ |x| \to \infty} \frac{v(x)}{1+|x|}.
\end{equation}
Then $u\leq v$ in $U$.
\end{thm}

We also give a result for equations with no zeroth-order term under a strictness condition. 

\begin{thm}\label{a.comp.strict}
Let $\theta > 0$, $U\subseteq \Rd$ be open and bounded, and $u,-v\in \USC\big(\overline U \big)$ satisfy
\begin{equation*}\label{}
 - \tr\left( A(x) D^2u\right) + H(Du,x) \leq -\theta < 0 \leq - \tr\left( A(x) D^2v\right) + H(Dv,x) \quad \mbox{in} \ U.
\end{equation*}
Then
\begin{equation*}\label{}
\sup_{U} \,(u-v) \leq \sup_{\partial U}\, (u-v).
\end{equation*}
\end{thm}

\begin{proof}[{\bf Proof of Theorem~\ref{a.comp.mono}}]

We argue by contradiction: assume that $u > v$ at some point of~$U$, which by translation we may assume to be the origin.

{\it Step 1.} We setup the argument.  Denote $\theta :=u(0)-v(0)>0$. Fix  $\eta, s, \ep>0$ satisfying
\begin{equation*} \label{}
0< \eta < \frac14 \theta, \quad \frac12 < s < 1 \quad \mbox{and} \quad 0< \ep < 1.
\end{equation*}
In the course of the argument, we will send $\ep \to 0$, $\eta \to 0$ and then $s \to 1$, in that order. Throughout we display the dependence of the constants on these parameters.

Consider the auxiliary function $\Phi:\overline U \times \overline U \to \R$ defined by
\begin{equation*} \label{}
\Phi(x,y):= su(x) - v(y) -\frac{1}{2\ep}|x-y|^2 - \eta\underbrace{ \left( 1+|x|^2\right)^{\frac12}}_{=:\phi(x)}.
\end{equation*}
If $\eta>0$ is sufficiently small and $s<1$ is sufficiently close to $1$, then 
\begin{equation*} \label{}
\sup_{U\times U} \Phi \geq \Phi(0,0) = su(0)-v(0) - \eta \geq \frac12\theta 
\end{equation*}
and therefore, by~\eqref{e.comp.growth} and the linear growth of $\phi$, there exist $(x_\ep,y_\ep) \in \overline U\times \overline U$ with $|x_\ep|, |y_\ep| \leq C_{\eta,s}$ such that 
\begin{equation} \label{e.ponk}
\Phi(x_\ep,y_\ep) = \sup_{U\times U} \Phi < +\infty.
\end{equation}
According to~\cite[Lemma~3.1]{CIL}, there exists $x_0\in \overline U$ with $|x_0| \leq C_{\eta,s}$ such that 
\begin{equation} \label{e.absurd}
su(x_0) - v(x_0) - \eta \left( 1+|x_0|^2\right)^{\frac12}  = \sup_{x\in U} \left( su(x) - v(x) - \eta \left( 1+|x|^2\right)^{\frac12}  \right)    \geq \frac12\theta
\end{equation}
and, up to a subsequence, 
\begin{equation} \label{e.xeyeplg}
\lim_{\ep \to 0} \, (x_\ep,y_\ep) = (x_0,x_0) \qquad \mbox{and} \qquad \lim_{\ep\to 0}\frac{|x_\ep-y_\ep |^2}{\ep}  = 0.
\end{equation}
If $s< 1$ is sufficiently close to $1$, then we have $x_0 \in U$ by~\eqref{e.absurd} and the first condition in~\eqref{e.comp.growth}, and therefore $(x_\ep,y_\ep) \in U\times U$ for sufficiently small $\ep> 0$. According to~\eqref{e.comp.deq},~\eqref{e.ponk} and the Crandall-Ishii Lemma~\cite[Lemma 3.2]{CIL}, there exist~$X_\ep, Y_\ep \in \Sy$ satisfying
\begin{equation}\label{suv-CI}
-\frac{3}{\ep}\begin{pmatrix} I_d &0\\ 0&I_d \end{pmatrix} \le \begin{pmatrix} X_\ep &0\\ 0&-Y_\ep \end{pmatrix} \leq \frac{3}{\ep}\begin{pmatrix} I_d &-I_d\\ -I_d&I_d \end{pmatrix}
\end{equation}
as well as
\begin{equation}\label{suv-test2}
\delta s u(x_\ep) - \tr \left (A(x_\ep) \left(X_\ep+\eta D^2 \phi(x_\ep) \right )\right) + sH\left( \frac{x_\ep-y_\ep}{\ep s} + \frac{\eta}{s} D\phi(x_\ep) ,\,x_\ep \right) \leq 0
\end{equation}
and
\begin{equation}\label{suv-test3}
\delta v(y_\ep) - \tr \left (A(y_\ep) Y_\ep\right ) + H\left( \frac{x_\ep-y_\ep}{\ep},y_\ep \right) \geq 0.
\end{equation}
The goal is to use~\eqref{e.absurd},~\eqref{suv-CI} and the structural conditions on $H$ to derive a contradiction by showing that the difference of the left sides of~\eqref{suv-test2} and~\eqref{suv-test3} must be positive after sending $\ep \to 0$, $\eta \to 0$, and then $s\to 1$.

{\it Step 2.} We estimate the difference between the terms involving $H$ on the left of~\eqref{suv-test2} and~\eqref{suv-test3}, respectively. This is the step of the argument which is unusual and departs from~\cite{CIL}; it relies on the convexity of $H$. It is convenient to set 
\begin{equation*} \label{}
p_\ep:=\frac{x_\ep-y_\ep}{\ep}, \qquad q_\ep:= \eta D\phi(x_\ep) = \eta \left(1+|x_\ep|^2\right)^{-\frac12}x_\ep.
\end{equation*}
Observe that $|q_\ep| \leq \eta$ and, by~\eqref{e.xeyeplg}, $|p_\ep |\cdot|x_\ep-y_\ep| \to 0$ as $\ep \to 0$. 

Fix $r:=\frac12(1+s) < 1$ and observe by the convexity of~$H$ that
\begin{equation}\label{e.quack}
H \left( \frac rs p_\ep ,x_\ep \right) \leq r H\left(\frac{p_\ep + q_\ep}{s},x_\ep \right)  + (1-r) H\left( -\frac{rq_\ep}{s(1-r)} ,x_\ep \right).
\end{equation}
Using 
\begin{equation*} \label{}
\frac{r|q_\ep|}{s(1-r)} \leq \frac{\eta (s+1)}{ s(1-s)} \leq \frac{4\eta}{1-s},
\end{equation*}
we may estimate the second term on the right of~\eqref{e.quack}  by~\eqref{e.Hsubq}. We get
\begin{equation}\label{e.quackers}
(1-r) H\left( -\frac{rq_\ep}{s(1-r)} ,x_\ep \right) \leq \frac12(1-s) \Lambda_1\left(4^m \eta^{\m}(1-s)^{-\m}+1 \right)  \leq C(1-s) + C_s\eta.
\end{equation}
Using~\eqref{e.quack},~\eqref{e.quackers} as well as~\eqref{e.Hsubq} and~\eqref{e.HsubqLip} with $R=\max\{ |x_\ep| , |y_\ep| \} \leq C_{\eta,s}$, we find that
\begin{align}
\lefteqn{s H\left(\frac{p_\ep + q_\ep}{s},x_\ep \right) - H\big(p_\ep,y_\ep) } \qquad & \label{e.danger} \\
& = s H\left(\frac{p_\ep + q_\ep}{s},x_\ep \right) -H\big(p_\ep,x_\ep) + H\big(p_\ep,x_\ep) - H\big(p_\ep,y_\ep) \nonumber \\
& \ge \underbrace{ \left( \frac sr H \left( \frac rs p_\ep ,x_\ep \right) - H(p_\ep,x_\ep)  \right)}_{=:G_\ep}
 - \big(\Lambda_1 |p_\ep|^\m +C_{\eta,s}\big) |x_\ep-y_\ep| - C(1-s) - C_s\eta.\nonumber
\end{align}
The dangerous term in~\eqref{e.danger} is $|p_\ep|^\m |x_\ep-y_\ep|$. Indeed, the control we have over $|p_\ep| \cdot |x_\ep-y_\ep|$ is not of immediate use since $\m>1$ and $|p_\ep|$ is expected to be large. To compensate, we use the convexity and growth of $H$ to show that the good term $G_\ep$ dominates the dangerous term in the limit $\ep \to 0$. Precisely, we claim that 
\begin{equation} \label{e.Gclaim}
G_\ep \geq c_{\eta,s}(1-s)|p_\ep|^\m - C(1-s).
\end{equation}
The important point is that the constants in this estimate do not depend on $\ep$, so that the right side of~\eqref{e.danger} is nonnegative after we send $\ep \to 0$, $\eta \to 0$ and then $s\to 1$. To prove~\eqref{e.Gclaim}, we set 
\begin{equation*} \label{}
\tau := \frac rs = \frac{s+1}{2s}.
\end{equation*}
Note that $1<\tau \leq \frac32$. We also fix $0<\beta<\frac12$ to be chosen below and observe that, due to the convexity of $H$, we have
\begin{equation*} \label{}
H(p_\ep,x_\ep) \leq (1-\lambda) H(\tau p_\ep,x_\ep) +  \lambda H(\beta p_\ep,x_\ep),
\end{equation*}
where we have defined
\begin{equation*}\label{}
\lambda := \frac{\tau-1}{\tau-\beta}, \quad \mbox{so that} \quad 1-\lambda = \frac{1-\beta}{\tau-\beta}.
\end{equation*}
Using~\eqref{e.Hsubq}, we obtain, for small $\ep>0$,
\begin{align*} \label{}
G_\ep  = \frac{1}{\tau} H(\tau p_\ep,x_\ep) - H(p_\ep,x_\ep)  & \geq \left( \frac1\tau +\lambda -1 \right) H(\tau p_\ep,x_\ep) - \lambda H(\beta p_\ep,x_\ep) \\
& = \frac{\tau-1}{\tau-\beta} \left( \frac\beta\tau H(\tau p_\ep,x_\ep) -  H(\beta p_\ep,x_\ep)\right)\\
& \geq \frac{\tau-1}{\tau-\beta} \left(  \frac\beta\tau \left(a_R\tau^\m |p_\ep|^\m -M_R\right) -  \Lambda_1 (\beta^\m |p_\ep|^\m+1)\right)\\
& = \frac{\tau-1}{\tau-\beta} \left( \frac\beta\tau a_R\tau^\m |p_\ep|^\m  -\Lambda_1 \beta^\m |p_\ep|^\m - \frac{\beta M_R}{\tau} - \Lambda_1\right),
\end{align*}
where~$R:=|x_0| +1\leq C_{\eta,s}$. Using that $1<\tau \leq \frac32$ and $0< \beta \leq \frac12$, we deduce that
\begin{equation*} \label{}
G_\ep \geq (\tau-1)\big( \beta \left( c_{\eta,s} |p_\ep|^\m\right)  - \beta^\m \left( C|p_\ep|^\m \right) - C_{\eta,s} \beta  - C\big).
\end{equation*}
Observe that~$(\tau-1) = (1-s)/2s > (1-s)/2$. We now see that, by choosing~$\beta > 0$ sufficiently small, depending $\eta$ and $s$, we obtain~\eqref{e.Gclaim}.

Combining~\eqref{e.danger} and~\eqref{e.Gclaim} with the fact~\eqref{e.xeyeplg} implies that $|x_\ep-y_\ep|\to 0$, we obtain that, for $0<\ep<\ep(\eta,s)$, 
\begin{equation} \label{e.diffH}
s H\left(\frac{p_\ep + q_\ep}{s},x_\ep \right) - H\big(p_\ep,y_\ep) \geq - C(1-s) - C_s\eta,
\end{equation}
 where $C>0$ does not depend on~$s,\ep,\eta$ and $C_s>0$, may depend on~$s$ but not on~$\ep$ or~$\eta$.

{\it Step 3.} We estimate the difference between the terms involving $A$ on the left of~\eqref{suv-test2} and~\eqref{suv-test3}, respectively. This step of the argument is just like in~\cite{CIL}. We proceed by multiplying the second inequality in \eqref{suv-CI} by the nonnegative matrix
\begin{equation*}
\begin{pmatrix}   \sigma^t(x_\ep) \sigma(x_\ep) &  \sigma^t(x_\ep) \sigma(y_\ep)  \\  \sigma^t(y_\ep) \sigma(x_\ep)  &\sigma^t(y_\ep) \sigma(y_\ep) \end{pmatrix}
\end{equation*}
and then take the trace of both sides to get
\begin{multline}\label{suv-test5}
2\tr \left (A(x_\ep) X_\ep\right)-\tr \left (A(y_\ep) Y_\ep\right) \\ \leq \frac3\ep \tr \left( (\sigma(x_\ep)-\sigma(y_\ep))^t(\sigma(x_\ep)-\sigma(y_\ep))\right)  \leq \frac{3\Lambda_2^2} \ep|x_\ep-y_\ep|^2.
\end{multline}
Observe that 
\begin{equation*}
D^2\phi(x_\ep) = \left(1+|x_\ep|^2\right)^{-\frac32} \left( (1+|x_\ep|^2)I_d-x_\ep \otimes x_\ep \right)
\end{equation*}
and therefore $|D^2\phi(x_\ep)| \leq C$, and so from~\eqref{suv-test5} we obtain
\begin{equation} \label{e.diffA}
- \tr \left (A(x_\ep) \left(X_\ep+\eta D^2 \phi(x_\ep) \right )\right) +  \tr \left (A(y_\ep) Y_\ep\right ) \leq \frac{C}{\ep} |x_\ep-y_\ep|^2 + C\eta.
\end{equation}

{\it Step 4.} We complete the argument. Subtracting~\eqref{suv-test3} from~\eqref{suv-test2} and inserting~\eqref{e.diffH} and~\eqref{e.diffA}, we discover that, for all $0<\ep < \ep(\eta,s)$,
\begin{equation*} \label{}
0 \geq \delta (su(x_\ep) - v(y_\ep))  - C(1-s) - \frac{C}{\ep} |x_\ep-y_\ep|^2 - C_s\eta.
\end{equation*}
In view of~\eqref{e.absurd} and~\eqref{e.xeyeplg}, we may let $\ep \to 0$ to get:
\begin{equation*} \label{}
0 \geq \frac12 \delta\theta  - C(1-s) - C_s\eta.
\end{equation*}
Sending $\eta \to 0$ and then $s\to 1$ yields the desired contradiction.
\end{proof}

\begin{proof}[{\bf Proof of Theorem \ref{a.comp.strict}}]
This is actually an easy consequence of Theorem~\ref{a.comp.mono}. Since $U$ is bounded, we may choose $\delta>0$ sufficiently small such that
\begin{equation*}\label{}
\delta u - \tr\left( A(x) D^2u\right) + H(Du,x) \leq -\frac\theta2 \leq \delta v - \tr\left( A(x) D^2v\right) + H(Dv,x) \quad \mbox{in} \ U,
\end{equation*}
and obtain the result from Theorem~\ref{a.comp.mono}.
\end{proof}

\subsection{Comparison principles for time-dependent problems}

We next give comparison principles for the time-dependent initial-value problems. The proof is similar to that of Theorem~\ref{a.comp.strict}, so we only sketch it.

\begin{thm}\label{Ca.comp}
Let $U \subseteq \Rd$ be open and  $T>0$ be a given positive constant  and assume that  $u,-v \in \USC(\overline{U} \times [0,T))$ satisfy \begin{equation*}\label{}
 u_t - \tr\left( A(x) D^2u\right) + H(Du,x) \leq 0 \leq v_t - \tr\left( A(x) D^2v\right) + H(Dv,x) \quad \mbox{in} \ U  \times (0,T),
\end{equation*}
and 
\begin{equation*}
u(\cdot,0) \le v(\cdot,0) \quad \text{on} \ \overline U,  \quad
u \le v \quad \text{on} \ \partial U \times [0,T)
\end{equation*}
and
\begin{equation*}
\limsup_{x \in U, |x| \to \infty} \sup_{t\in [0,T)} \frac{u(x,t)}{1+|x|} \le 0 \le \liminf_{x\in U, |x| \to\infty} \inf_{t\in [0,T)} \frac{v(x,t)}{1+|x|}.
\end{equation*}
Then $u \le v$ in $U \times [0,T)$.
\end{thm}
\begin{proof}
Observe first that $\bar u:=u-\alpha(T-t)^{-1}$ for $\alpha>0$ is a subsolution of
\begin{equation}\label{C-comp-ep}
\bar u_t - \tr\left( A(x) D^2\bar u\right) + H(D\bar u,x) \leq -\frac{\alpha}{(T-t)^2}\le -\frac{\alpha}{T^2} \quad \text{in} \ U \times (0,T).
\end{equation}
In order to prove $u \le v$, it is enough to prove that $\bar u \le v$ for any $\alpha>0$. We can therefore prove the comparison principle under the additional assumptions that
\begin{equation}\label{C-comp-mod}
\begin{cases}
u_t - \tr\left( A(x) D^2 u\right) + H(D u,x) \leq - \alpha T^{-2}=:-c \quad \text{in} \ U \times (0,T),\\
\lim_{t \to T} u(x,t)=-\infty \quad \text{uniformly on} \ \overline U.
\end{cases}
\end{equation}
We follow the same strategy as in the proof of Theorem~\ref{a.comp.mono}. Suppose by contradiction that $u>v$ at some point $(x_1,t_1) \in U \times (0,T)$. Set $\theta:=u(x_1,t_1)-v(x_1,t_1)>0$. Fix $\frac12<s<1$ and $\ep,\eta>0$ and define the auxiliary function $\Phi:\overline U \times \overline U \times [0,T) \to \R$ by
\begin{equation*}
\Phi(x,y,t):= su(x,t) - v(y,t) -\frac{1}{2\ep}|x-y|^2 - \eta\underbrace{ \left( 1+|x|^2\right)^{\frac12}}_{=:\phi(x)}.
\end{equation*}
Then,  if $\eta,\ep> 0$ are sufficiently small and $s<1$ is sufficiently close to $1$, there exist $(x_0,t_0)\in U \times (0,T)$ with $|x_0| \leq C_\eta$ and $(x_\ep,y_\ep,t_\ep) \in U \times U \times [0,T)$ such that 
\begin{equation*} \label{}
\Phi(x_\ep,y_\ep,t_\ep) = \sup_{ U\times U \times (0,T)} \Phi < +\infty,
\end{equation*}
\begin{equation} \label{C.absurd}
\sup_{x\in U  \times [0,T)} \left( su(x,t) - v(x,t) - \eta \left( 1+|x|^2\right)^{\frac12}  \right)  = su(x_0,t_0) - v(x_0,t_0) - \eta \left( 1+|x_0|^2\right)^{\frac12}  \geq \frac12\theta
\end{equation}
and, up to a subsequence, 
\begin{equation} \label{C.xeyeplg}
\lim_{\ep \to 0} (x_\ep,y_\ep,t_\ep) = (x_0,x_0,t_0) \qquad \mbox{and} \qquad \lim_{\ep\to 0}\frac{|x_\ep-y_\ep |^2}{\ep}  = 0.
\end{equation}
Note that $\Phi(x_\ep,y_\ep,t_\ep) \ge \frac12 \theta$, which together with \eqref{C.xeyeplg} implies that $t_\ep>0$ for $\ep>0$ sufficiently small.

By~\cite[Theorem 8.3]{CIL}, there exist $\tau\in \R$ and symmetric matrices $X_\ep, Y_\ep \in \Sy$ satisfying
\begin{equation}\label{time-CI}
-\frac{3}{\ep}\begin{pmatrix} I_d &0\\ 0&I_d \end{pmatrix} \le \begin{pmatrix} X_\ep &0\\ 0&-Y_\ep \end{pmatrix} \leq \frac{3}{\ep}\begin{pmatrix} I_d &-I_d\\ -I_d&I_d \end{pmatrix},
\end{equation}
as well as
\begin{equation}\label{time-test1}
 s \tau - \tr \left (A(x_\ep) \left(X_\ep+D^2 \phi(x_\ep) \right )\right) + sH\left( \frac{x_\ep-y_\ep}{\ep s} + \frac{\eta}{s} D\phi(x_\ep) ,\,x_\ep \right) \leq -sc.
\end{equation}
and
\begin{equation}\label{time-test2}
\tau - \tr \left (A(y_\ep) Y_\ep\right ) + H\left( \frac{x_\ep-y_\ep}{\ep},y_\ep \right) \geq 0.
\end{equation}
Follow the last step in the proof of Theorem~\ref{a.comp.mono}, we obtain from~\eqref{time-test1} and~\eqref{time-test2} the desired contradiction that $0\le -c$.
\end{proof}

\subsection{Two useful consequences of convexity and comparison}

In this subsection we prove two simple lemmas, used repeatedly in the rest of the paper, involving convex combinations of subsolutions and supersolutions of the equations
\begin{equation} \label{e.convcomb}
-\tr(A(x)D^2u) + H(Du,x) = \mu  \quad \mbox{and}\quad -\tr(A(x)D^2v) + H(Dv,x) = \nu.
\end{equation}

\begin{lem}
\label{l.convex}
Suppose that $U \subseteq\Rd$ is open, $\mu,\nu\in \R$ and $u,v\in\USC(U)$ such that $u$ and $v$ are subsolutions of the equations in~\eqref{e.convcomb} respectively. Then for each $0 \leq \lambda \leq 1$, the function $w:=\lambda u + (1-\lambda)v$ is a subsolution of
\begin{equation} \label{e.convex1}
-\tr(A(x)D^2w) + H(Dw,x) = \lambda \mu + (1-\lambda) \nu \quad \mbox{in} \ U.
\end{equation}
\end{lem}

\begin{lem}
\label{l.convex2}
Suppose that $U \subseteq\Rd$ is open, $\mu,\nu\in \R$ and $u,-v\in\USC(U)$ such that $u$ is a subsolution and $v$ is a supersolution of the equations in~\eqref{e.convcomb} respectively. Then for every $\lambda \geq 0$, the function $w:=(1+\lambda) v -\lambda u$ is a supersolution of
\begin{equation} \label{e.convex2}
-\tr(A(x)D^2w) + H(Dw,x) = (1+\lambda) \nu -\lambda \mu \quad \mbox{in} \ U.
\end{equation}
\end{lem}

The statements of both lemmas are easy to formally derive from the convexity of $H$. If one of $u$ or $v$ is $C^2$, then the formal derivation is actually rigorous because it can be repeated by using a test function in place of the other, nonsmooth function. All of the interest is therefore in the case when neither of $u$ or $v$ is smooth.

\begin{proof}[{\bf Proof of Lemma~\ref{l.convex}}]
We assume $0<\lambda<1$, otherwise there is nothing to prove. Assume that 
\begin{equation}\label{l.co-1}
w-\psi \quad \text{has a strict local maximum at} \quad x_0\in U
\end{equation}
 for some smooth test function $\psi$, and we need to show that 
\begin{equation}\label{l.co-2}
-\tr\left(A(x_0)D^2 \psi(x_0)\right)+H(D\psi(x_0),x_0)\le \lambda \mu+(1-\lambda)\nu.
\end{equation}
Suppose by contrary that there exists $r>0$ such that
\begin{equation}\label{l.co-3}
-\tr\left(A(x)D^2 \psi(x))+H(D\psi(x),x\right)> \lambda \mu+(1-\lambda)\nu \quad \text{for every}\ x \in B_r(x_0).
\end{equation}
We claim that  $\tilde v := (1-\lambda)^{-1}(\psi-\lambda u)$ is a solution of
\begin{equation}\label{l.co-4}
-\tr\left (A(x)D^2 \tilde v \right)+H(D\tilde v,x)> \nu \quad \text{in}\ B_r(x_0).
\end{equation}
To check~\eqref{l.co-4} in the viscosity sense, we take a smooth test function $\varphi$ such that $\tilde v -\varphi$ has a strict local minimum at $x_1 \in B(x_0,r)$. Then $u-\lambda^{-1}\left(\psi-(1-\lambda) \varphi \right)$ has a strict local maximum at $x_1$. Using the differential inequality for $u$, we find
\begin{multline}\label{l.co-5}
-\tr\left ( \lambda^{-1} A(x_1) \left(D^2\psi(x_1)-(1-\lambda)D^2 \varphi(x_1)\right)\right ) \\ +H\left ( \lambda^{-1} \left(D\psi(x_1)-(1-\lambda)D \varphi(x_1)\right),x_1\right ) \le \mu.
\end{multline}
Combining \eqref{l.co-3},~\eqref{l.co-5} and the convexity of $H$ yields
\begin{equation}
-\tr\left ( A(x_1)D^2\varphi(x_1)\right )+H(D\varphi(x_1),x_1) >\nu.
\end{equation}
This confirms \eqref{l.co-4}. We apply Proposition~\ref{a.comp.strict} to conclude that 
\begin{equation*}
\min_{\overline B_r(x_0)} (\tilde v -v ) = \min_{\partial B_r(x_0)} (\tilde v - v),
\end{equation*}
which contradicts  \eqref{l.co-1}. The proof is complete.
\end{proof}

To formally derive Lemma~\ref{l.convex2}, write $v$ in terms of $u$ and $w$:
\begin{equation*}\label{}
v = \frac{1}{1+\lambda} w + \frac{\lambda}{1+\lambda} u = \frac{1}{1+\lambda} w + \left( 1 - \frac{1}{1+\lambda} \right) u.
\end{equation*}
Observe that, since~$u$ is a subsolution of the first equation of~\eqref{e.convcomb}, Lemma~\ref{l.convex} asserts that~$v$ is a strict subsolution of the second equation of~\eqref{e.convcomb} anywhere in~$U$ that~$w$ is a strict subsolution of~\eqref{e.convex2}. But by hypothesis~$v$ is a supersolution in~$U$, thus it cannot be a strict subsolution anywhere in~$U$, and so we conclude that~$w$ cannot be a strict subsolution of~\eqref{e.sc} anywhere in~$U$. Thus~$w$ must be a supersolution of~\eqref{e.sc} in~$U$. We now make this argument rigorous.

\begin{proof}[{\bf Proof of Lemma~\ref{l.convex2}}]
Select a smooth test function $\phi$ and a point $x_0\in U$ such that 
\begin{equation}\label{e.sc.subq1.ctr}
w - \phi \quad \mbox{has a strict local minimum at} \ x_0.
\end{equation}
We must show that
\begin{equation*}\label{}
-\tr\left( A(x_0) D^2\phi(x_0) \right) + H(D\phi(x_0),x_0) \geq (1+\lambda) \nu - \lambda \mu.
\end{equation*}
Arguing by contradiction, we assume on the contrary that 
\begin{equation*}\label{}
\theta:= (1+\lambda) \nu - \lambda \mu +\tr\left( A(x_0) D^2\phi(x_0) \right) - H(D\phi(x_0),x_0) > 0.
\end{equation*}
Since $\phi$ is smooth, we may select $r>0$ sufficiently small that $\overline B_r(x_0) \subseteq U$ and
\begin{equation}\label{e.sc.subq1.psi}
-\tr\left( A(x) D^2\phi \right) + H(D\phi,x) \leq  (1+\lambda) \nu - \lambda \mu- \frac12 \theta \quad \mbox{in} \ B_r(x_0).
\end{equation}
Note that~\eqref{e.sc.subq1.ctr} is equivalent to
\begin{equation}\label{e.sc.subq1.ctr2}
v -  \underbrace{\left(  \frac{1}{1+\lambda}\phi + \frac{\lambda}{1+\lambda} u \right) }_{=:\phi_\lambda}  \quad \mbox{has a strict local minimum at} \ x_0.
\end{equation}
According to Lemma~\ref{l.convex} and~\eqref{e.sc.subq1.psi}, the function $\phi_\lambda$ satisfies
\begin{equation}\label{sc.subq1.psi2}
-\tr\left( A(x) D^2\phi_\lambda \right) + H(D\phi_\lambda,x) \leq \frac{1}{1+\lambda}\left((1+\lambda) \nu - \lambda \mu- \frac12 \theta \right) + \frac{\lambda}{1+\lambda} \mu < \nu \quad \mbox{in} \ B_r(x_0).
\end{equation}
We now have the desired contradiction, since the fact that $v$ is a supersolution of the second equation of~\eqref{e.convcomb} is in violation of~\eqref{e.sc.subq1.ctr2},~\eqref{sc.subq1.psi2} and Proposition~\ref{a.comp.strict}. 
\end{proof}

\begin{remark}
Left to the reader are analogues of Lemmas~\ref{l.convex} and~\ref{l.convex2} with nearly identical proofs: for time-dependent equations, with nonconstant right-hand sides, and so forth.
\end{remark}

\section{Interior H\"older and Lipschitz estimates}
\label{S.A}

\subsection{Interior Lipschitz estimates: The stationary case.} \textbf{}
The purpose of this subsection is to prove the following explicit Lipschitz bound for \emph{continuous} viscosity solutions. We do not use the assumption that~$H$ is convex here, so the results hold also for nonconvex~$H$.

\begin{thm}
\label{t.LIP}
Fix $\delta \geq 0$. Then any solution $u\in C(B_2)$ of
\begin{equation} \label{e.uLipPde}
\delta u-\tr\left( A(x) D^2u \right) + H(Du,x) = 0 \quad \mbox{in} \ B_2
\end{equation}
satisfies, for every $x,y\in B_1$,
\begin{equation} \label{}
\left| u(x) - u(y) \right| \leq  K |x-y|,
\end{equation}
where $K>0$ is given by
\begin{equation} \label{e.Lmu}
K :=  C \left\{ \left( \frac{(1+\Lambda_1)^{1/2}\Lambda_2}{a_2} \right)^{2/(m-1)}+\left( \frac{M_2 + \delta \| u\|_{L^\infty(B_{3/2})} }{a_2}\right)^{1/m}  \right\}
\end{equation}
and $C >0$ depends only on $\d$ and $m$. In particular, $\osc_{B_1} u \leq 2K$.
\end{thm}

The proof of Theorem~\ref{t.LIP} is by \emph{Bernstein's method}~\cite{Bern}, which is a technique for deriving \emph{a priori}~$L^\infty$ bounds on the gradient of a solution of an elliptic PDE by differentiating the equation and applying the maximum principle. The main idea is that a power of $|Du|$ should be a subsolution of an elliptic equation, so the maximum principle forbids $|Du|$ from attaining a local maximum. For example, if $u$ is harmonic, then $|Du|^2$ is subharmonic. Modifying the argument by inserting appropriate cutoff functions, one can deduce that $|Du|$ cannot be large away from the boundary of the underlying domain; that is, the technique yields $L^\infty$ bounds on $Du$.

The fact that we work with viscosity solutions complicates the details of the Bernstein argument, since we only assume \emph{a priori} that the solution is continuous (but we must assume that it is continuous-- the argument does not work for discontinuous solutions and indeed the result is false). The proof in our setting (with similar structural conditions) in the case that everything is smooth can be found for example in~\cite[Lemma~4.8]{AS1}. Barles~\cite{Ba}  was the first to implement a modification of Bernstein's method in the framework of viscosity solutions. The idea is that, since we cannot assume $Du$ exists in any useful sense, rather than differentiating the equation and applying the comparison principle in two distinct steps, we must differentiate the equation ``inside the proof" of the comparison principle.

Local Lipschitz estimates like the one contained in Theorem~\ref{t.LIP} are well-known (c.f.~\cite{CDLP}), and we include the proof here for two reasons. First, the argument presented here is new and we find it to be less involved and more straightforward than others we could find in the literature, which do not exactly match our assumptions. Second, due to the efficiency of the argument, we are able to derive an explicit Lipschitz constant which, in terms of the other parameters in the structural hypotheses, is sharp. This explicit estimate plays an important role in the analysis for the stochastic homogenization of these equations, see~\cite{AC,AT}. 

Before giving the proof of Theorem~\ref{t.LIP}, we first recall that, in the superquadratic case, $m>2$, we obtain a H\"older estimate for \emph{subsolutions} (which may be \emph{a priori} discontinuous) using only the scaling of the equation. This is because, for a superquadratic Hamilton-Jacobi equation, the second-order term is of secondary importance to the strongly coercive Hamiltonian on small length scales and the equation  behaves in certain respects like a first-order equation. This was previously observed for example in~\cite{PLL,LL,CDLP}. We use this estimate in the proof of Theorem~\ref{t.LIP}.

The proof of the H\"older bound is simple: we exhibit an explicit, smooth supersolution in a punctured ball which blows up on the boundary of the ball.

\begin{lem}
\label{l.superquad}
Assume that $\m > 2$ and suppose that $u\in \USC(B_2)$ satisfies
\begin{equation}\label{}
-\Lambda_2^2\! \left|D^2u\right| +a_2|Du|^\m  \leq M_2 \quad \mbox{in} \ B_2.
\end{equation} 
Then, for every $x,y\in B_1$,
\begin{equation} \label{e.superquad.hod}
 |u(x) - u(y)| \leq K |x-y|^{\gamma}, \quad \mbox{where} \quad \gamma:= \frac{m-2}{m-1}
\end{equation}
and $K>0$ is given explicitly by
\begin{equation*}
K:= C \left( \left( \frac{\Lambda_2^2}{a_2}\right)^{\frac1{\m-1}} + \left( \frac {M_2}{a_2}\right)^{\frac1\m}  \right)
\end{equation*}
and $C>0$ depends only on $\d$ and $\m$.
\end{lem}
\begin{proof}
Observe that $0 < \gamma <1$ and $\m(\gamma-1) = \gamma-2$. Also fix $x\in B_1$. We claim that, for an appropriate constant $K> 0$, the function 
\begin{equation*}\label{}
\phi(y):=  K \left( 1-|y-x|^2 \right)^{-1} |y-x|^{\gamma}
\end{equation*}
is a smooth solution of
\begin{equation}\label{e.phidemo}
-\Lambda_2^2|D^2\phi| + a_2 |D\phi|^\m  > M_2 \quad \mbox{in} \ B_1(x) \setminus \{ x \}. 
\end{equation}
We first show that~\eqref{e.phidemo} implies~\eqref{e.superquad.hod}. It is immediate from~\eqref{e.phidemo} and the definition of viscosity subsolution that the function $u(\cdot) - u(x) - \phi(\cdot)$ has no local maximum in $B_1 (x)\setminus \{ x \}$. Since $\phi(y)$ blows up as $y \to \partial B_1(x)$, we deduce the supremum of this function in $B_1(x)$ is achieved at $x$, where it vanishes. Hence it is nonpositive in $B_1(x)$, and we obtain, for every $y\in B_{1/2}(x)$,
\begin{equation*}\label{}
u(y) - u(x)  \leq \phi(y) = K \left( 1-|y-x|^2 \right)^{-1} |y-x|^{\gamma} \leq \frac{4}{3} K |y-x|^{\gamma}. 
\end{equation*}
The triangle inequality then gives the desired estimate.

The verification of~\eqref{e.phidemo} is just a routine calculation. Since~\eqref{e.phidemo} is transition invariant, we may suppose $x=0$. We compute
\begin{align*}\label{}
\left| D\phi(y) \right|^\m & = K^\m \left( 1-|y|^2 \right)^{-2\m} |y|^{m(\gamma-1)} \left( 2|y|^2+ \gamma\left(1-|y|^2\right) \right)^m \\
& \geq K^\m \left( 1-|y|^2 \right)^{-2\m} |y|^{\gamma-2} \gamma^m
\end{align*}
and, for a constant $C> 0$ depending only on $\d$, 
\begin{equation*}\label{}
\left| D^2 \phi(y) \right| \leq C K \left( 1-|y|^2 \right)^{-3} \left| y \right|^{\gamma-2}.
\end{equation*}
Therefore, for $C,c>0$ depending on $d$ and $m$,
\begin{align*} \label{}
-\Lambda_2^2|D^2\phi| + a_2 |D\phi|^\m & \geq \left(   -CK \Lambda_2^2  + c a_2 K^m\left( 1-|y|^2 \right)^{3-2\m} \right) \left( 1-|y|^2 \right)^{-3} \left| y \right|^{\gamma-2} \\
& \geq \left(   -CK \Lambda_2^2  + c a_2 K^m\right) \left( 1-|y|^2 \right)^{-3} \left| y \right|^{\gamma-2}.
\end{align*}
We now select $K>0$ large enough that the term in the first parentheses is at least $M_2$. It suffices to take 
\begin{equation*} \label{}
K:= C\left( \left( \frac{\Lambda_2^2}{a_2}\right)^{\frac1{\m-1}} + \left( \frac {M_2}{a_2}\right)^{\frac1\m}  \right),
\end{equation*}
for $C>0$ depending on $d$ and $m$. Since the $\left( 1-|y|^2 \right)^{-3} \left| y \right|^{\gamma-2} >1$ in $B_1\setminus\{ 0 \}$, we obtain~\eqref{e.phidemo}, as desired.
\end{proof}

We now give the proof of the Lipschitz estimate.

\begin{proof}[{\bf Proof of Theorem~\ref{t.LIP}}]
For the sake of clarity, we first give the argument without using a cutoff function, which makes it much easier to follow the underlying ideas. This is done in Step~0. We then give the complete proof, beginning in Step~1, by modifying this argument to include the cutoff function.
 
{\it Step 0.} We suppose that $L> 0$ and $x_0,y_0\in B_1$ such that 
\begin{equation}\label{e.FPL}
u(x_0) - u(y_0) - L|x_0-y_0| = \sup_{x,y\in B_1} \left(u(x)-u(y)-L|x-y| \right) > 0 
\end{equation}
and argue that $L > 0$ cannot be too large. The fact that the supremum in~\eqref{e.FPL} is attained at some $x_0,y_0 \in B_1$ is an unjustified assumption removed below via the use of a cutoff function.

Proceeding with the argument, we observe that since the supremum in~\eqref{e.FPL} is positive, we must have $x_0 \neq y_0$. As $u$ is a solution of~\eqref{e.uLipPde}, the Crandall-Ishii lemma~\cite[Lemma 3.2]{CIL} yields, for each $\ep > 0$,  matrices $X_\ep, Y_\ep \in \Sy$ which satisfy the matrix inequality
\begin{equation}\label{e.mats}
\begin{pmatrix} X_\ep &0\\ 0&-Y_\ep \end{pmatrix} \leq J + \ep J^2, 
\end{equation}
where we have defined the matrices
\begin{equation}\label{e.JandX}
 J:= \frac{L}{|x_0-y_0|} \begin{pmatrix} Z & -Z \\ -Z& Z \end{pmatrix} \quad \mbox{and} \quad  Z:= I_d - \frac{x_0-y_0}{|x_0-y_0|} \otimes \frac{x_0-y_0}{|x_0-y_0|},
\end{equation}
as well as
\begin{multline}\label{e.PDEs}
\delta u(x_0)-\tr\left( A(x_0) X_\ep \right) + H\left( L\frac{x_0-y_0}{|x_0-y_0|} , x_0 \right) \leq 0 \\ \leq \delta u(y_0)-\tr\left( A(y_0) Y_\ep \right) + H\left( L\frac{x_0-y_0}{|x_0-y_0|} , y_0 \right).
\end{multline}
The rest of the argument is concerned with deriving a bound for $L$ from~\eqref{e.mats} and~\eqref{e.PDEs}.

Fix $s>1$ to be selected below. Multiplying both sides of the matrix inequality~\eqref{e.mats} on the right by the nonnegative matrix
\begin{equation}\label{}
A_s:= \frac12 \begin{pmatrix} s^2  \sigma^t(x_0) \sigma(x_0) & s \sigma^t(x_0) \sigma(y_0)  \\ s \sigma^t(y_0) \sigma(x_0)  &\sigma^t(y_0) \sigma(y_0) \end{pmatrix} = \frac12 \begin{pmatrix}  s \sigma(x_0) \\ \sigma(y_0) \end{pmatrix}^t \begin{pmatrix}  s \sigma(x_0) \\ \sigma(y_0) \end{pmatrix} \geq 0
\end{equation}
and taking the trace of the resulting expression, we get
\begin{equation}\label{e.pause}
\tr\left( s^2A(x_0) X_\ep - A(y_0) Y_\ep \right) \leq \tr\left( J A_s \right) + \ep \tr\left( J^2 A_s \right). 
\end{equation}
We next estimate the right side of the last expression (since we are going to set $\ep \to 0$ eventually, we ignore the second term). We have:
\begin{align*}
\tr\left( J A_s \right) & = \frac{L}{2|x_0-y_0|} \tr\big(   (s\sigma(x_0)-\sigma(y_0)) \,Z\,(s\sigma(x_0)-\sigma(y_0))^t\big) & \mbox{(computation)}\\ & \leq  \frac{L}{2|x_0-y_0|}\tr\big(  (s\sigma(x_0)-\sigma(y_0)) (s\sigma(x_0)-\sigma(y_0))^t\big) & \mbox{($Z^2=Z$)} \\
& = \frac{L}{2|x_0-y_0|} \left| s\sigma(x_0) - \sigma(y_0) \right|^2 \\
& \leq \frac{L}{2|x_0-y_0|} \left(  \Lambda_2|s-1| + \Lambda_2|x_0-y_0| \right)^2  & \mbox{(\eqref{e.sigbnd} \& \eqref{e.siglip})} \\
& \leq \frac{L \Lambda_2^2}{|x_0-y_0|} \left( (s-1)^2 + |x_0-y_0|^2  \right).  & \mbox{(Cauchy ineq.)}
\end{align*}
Setting $s^2:= 1 + \beta |x_0-y_0|$ with $\beta > 0$ to be selected below, and noticing that $(s-1)^2 \leq (s+1)^2(s-1)^2 = \beta^2|x_0-y_0|^2$, this expression simplifies to
\begin{equation}\label{e.pauseR}
\tr\left( J A_s \right) \leq L\Lambda_2^2 \left( 1+ \beta^2 \right) |x_0-y_0|.
\end{equation}
We are done with the right side of~\eqref{e.pause} and we proceed to estimate its left side from below:
\begin{align*}
\lefteqn{\tr\left( s^2A(x_0) X_\ep - A(y_0) Y_\ep \right)} 
\qquad \qquad & \\ & \geq s^2 H\left( L\frac{x_0-y_0}{|x_0-y_0|} , x_0 \right) - H\left( L\frac{x_0-y_0}{|x_0-y_0|} , y_0 \right) 
+\delta(s^2u(x_0)-u(y_0))& \mbox{(\eqref{e.PDEs})}\\
& \geq (s^2-1)  H\left( L\frac{x_0-y_0}{|x_0-y_0|} , x_0 \right)  - \left( \Lambda_1 L^\m+M_2 \right) |x_0-y_0|
+\delta (s^2-1)u(x_0) & \mbox{(\eqref{e.HsubqLip})}\\
& \geq (s^2-1) \left( a_2 L^\m - M_2  \right)  - \left( \Lambda_1 L^\m+M_2 \right) |x_0-y_0|
+\delta (s^2-1)u(x_0).& \mbox{(\eqref{e.Hsubq})}
\end{align*}
Inserting $s^2:=1+\beta|x_0-y_0|$, comparing with~\eqref{e.pause} and~\eqref{e.pauseR} and dividing by $|x_0-y_0|$, we obtain
\begin{equation*}\label{}
\beta a_2 L^m \leq \beta \left( M_2 -\delta u(x_0) \right) + \Lambda_1L^m + M_2 + L \Lambda_2^2\left(1+\beta^2\right) + \ep \tr\left( J^2 A_s \right).
\end{equation*}
Sending $\ep \to 0$ and a slight rearrangement give
\begin{equation*}\label{}
L^\m \left( \beta a_2 - \Lambda_1 \right) \leq \beta \left( M_2 -\delta u(x_0) \right) + M_2 + L \Lambda_2^2\left(1+\beta^2\right).
\end{equation*}
We now choose $\beta := (1+2\Lambda_1)/a_2$. This gives
\begin{align*}\label{}
L^\m (1+\Lambda_1) \leq \frac{2}{a_2} \left( M_2-\delta u(x_0) \right) \left( 1 + \Lambda_1 \right) + \frac{1}{a_2^2}L\Lambda_2^2 \left( 3+8\Lambda_1^2 \right).
\end{align*}
We now use the elementary fact that 
\begin{equation*} \label{}
x\geq 2^{1/(m-1)} \left( a^{1/(m-1)} + b^{1/m} \right) \quad \implies \quad x^m \geq ax +b,
\end{equation*}
to deduce that 
\begin{equation*}\label{}
L \leq 2^{1/(m-1)} \left( \left( \frac{8\Lambda_2^2(1+\Lambda_1)}{a_2^2}  \right)^{1/(m-1)} + \left( \frac{2(M_2+\delta |u(x_0)|)}{a_2} \right)^{1/m}   \right).
\end{equation*}

{\it Step 1.} 
We now begin the general (rigorous) argument, addressing the problem that the supremum in~\eqref{e.FPL} may not be attained in general. We begin with the observation that Lipschitz continuity is a \emph{local} property. That is, to prove $\sup_{x,y\in B_1} |u(x) - u(y)| \leq L|x-y|$, it suffices to show that, for every $\hat x\in B_1$, 
\begin{equation}\label{e.wtsliploc}
\limsup_{x \to \hat x}  \frac{u(\hat x) - u(x)}{|\hat x-x|} \leq L. 
\end{equation}
We therefore proceed by fixing~$L\geq 1$ and~$\hat x \in B_1$ such that 
\begin{equation}\label{e.liplocass}
\limsup_{x \to\hat x} \frac{ u(\hat x) - u(x)}{|\hat x-x|} > L 
\end{equation}
and derive a contradiction by taking $L$ to be too large. 

Playing the role of the cutoff function is a positive smooth function $\phi :  B_{3/2} \to [1,\infty)$ which satisfies $\phi \equiv 1$ on $B_1$ and $\phi(x) \to +\infty$ as $|x| \to \partial  B_{3/2}$. We also take $\phi$ so that, for each $x \in  B_{3/2}$,
\begin{equation}\label{e.phi}
\left|D\phi(x) \right| \leq C \left( \phi(x) \right)^\m \quad \mbox{and} \quad \left| D^2\phi(x) \right| \leq C \left( \phi(x) \right)^{2\m -1}.
\end{equation}
A regularization of the map $x\mapsto \max\left\{  \left (2 \dist\left(x,\partial  B_{3/2}\right)\right)^{-\frac{1}{\m-1}}, 1\right\}$ will do.

For each $\alpha>0$ sufficiently small, there exist points $x_\alpha,y_\alpha\in  B_{3/2}$ which satisfy
\begin{multline}\label{e.FPL2}
u(x_\alpha) - u(y_\alpha) - L\phi(y_\alpha) |x_\alpha-y_\alpha| - \frac1{2\alpha} |x_\alpha-y_\alpha|^2 \\
= \sup_{x,y\in  B_{3/2}} \left( u(x)-u(y)-L\phi(y) |x-y| - \frac1{2\alpha} |x-y|^2\right) > 0.
\end{multline}
Here is the reason: that the supremum in~\eqref{e.FPL2} is positive for each $\alpha> 0$ is due to~\eqref{e.liplocass}; the fact that there are points $x_\alpha,y_\alpha\in  B_{3/2}$ which attain this supremum, for sufficiently small $\alpha> 0$, is due to the uniform continuity of $u$ on $\overline B_{3/2}$, the positivity of the supremum, and the fact that~$\phi$ penalizes points which are too close to $\partial  B_{3/2}$. Indeed, the positivity of the supremum ensures that $y_\alpha\neq x_\alpha$, and in fact $|x_\alpha-y_\alpha|$ is bounded below by a positive constant in terms of the uniform continuity of~$u$. Since $|x_\alpha-y_\alpha|$ cannot be too small, the presence of $\phi$ ensures that $y_\alpha$ is kept away from $\partial  B_{3/2}$; since $u$ is bounded on $\overline B_{3/2}$, the quadratic term ensures that $|x_\alpha-y_\alpha|$ is also small:
\begin{equation} \label{e.dreg}
 \frac1{2\alpha} |x_\alpha-y_\alpha|^2 \leq \osc_{ B_{3/2}} u < +\infty.
\end{equation}
Therefore $x_\alpha$ is close to $y_\alpha$ and thus away from $\partial  B_{3/2}$, for sufficiently small $\alpha>0$.

Using to~\eqref{e.FPL2}, the continuity of $u$ and~\eqref{e.dreg}, we have
\begin{multline}\label{e.capt}
\limsup_{\alpha\to 0} \left( L\phi(y_\alpha) |x_\alpha-y_\alpha| +  \frac1{2\alpha} |x_\alpha-y_\alpha|^2 \right) \\
\leq \limsup_{\alpha\to 0} \, \sup \left\{ u(y) - u(z) \, : \, y,z\in B_{3/2},  \ |y-z| \leq \alpha^{\frac12} \osc_{B_{3/2}}u  \right\}  = 0.
\end{multline}
In the case that $m>2$, we may apply Lemma~\ref{l.superquad} to do better than~\eqref{e.capt}. We have
\begin{equation} \label{e.captm2}
L\phi(y_\alpha) |x_\alpha-y_\alpha| +\frac1{2\alpha} |x_\alpha-y_\alpha|^2 \leq u(x_\alpha) - u(y_\alpha) \leq \tilde K \left| x_\alpha - y_\alpha \right|^\gamma, 
\end{equation}
where $\gamma:=(m-2)/(m-1)$ and $\tilde K$ is the explicit constant $K$ in Lemma~\ref{l.superquad}. This implies in particular that 
\begin{equation} \label{e.PaQa}
(\phi(y_\alpha))^{m-1} \left| x_\alpha-y_\alpha\right| \leq L^{1-m}\tilde K^{m-1}, 
\end{equation}
which will be useful below. 

The rest of the argument is similar to Step~0; the differences due to the presence of $\phi$ and the quadratic term do not cause any real harm, only some bookkeeping headaches.

\emph{Step 2.} Applying the Crandall-Ishii lemma and using that $u$ is a solution of~\eqref{e.uLipPde}, we obtain, for each $\ep > 0$ and sufficiently small $\alpha> 0$, symmetric matrices $X_{\ep,\alpha}, Y_{\ep,\alpha}\in \Sy$ such that 
\begin{equation}\label{e.mats-g}
\begin{pmatrix} X_{\ep,\alpha} &0\\ 0&-Y_{\ep,\alpha} \end{pmatrix} \leq J_\alpha+ \ep J_\alpha^2, 
\end{equation}
and
\begin{equation}\label{e.PDEs-g}
\delta u(x_\alpha)-\tr\left( A(x_\alpha) X_{\ep,\alpha} \right) + H\left( \left( L \phi(y_\alpha) + \frac{|x_\alpha-y_\alpha|}{\alpha} \right) \frac{x_\alpha-y_\alpha}{|x_\alpha-y_\alpha|} , x_\alpha\right) \leq 0 
\end{equation}
\begin{multline}\label{e.PDEs-g2}
0 \leq \delta u(y_\alpha) -\tr\left( A(y_\alpha) Y_{\ep,\alpha} \right) \\+ H\bigg( \underbrace{\left( L \phi(y_\alpha) + \frac{|x_\alpha-y_\alpha|}{\alpha} \right) \frac{x_\alpha-y_\alpha}{|x_\alpha-y_\alpha|}}_{=:P_\alpha} - \underbrace{L|x_\alpha-y_\alpha| D\phi(y_\alpha)}_{=:Q_\alpha} , y_\alpha\bigg),
\end{multline}
where we have defined
\begin{equation}\label{}
J_\alpha:=\underbrace{\frac{L\phi(y_\alpha)}{|x_\alpha-y_\alpha|} \begin{pmatrix} Z_1&-Z_1\\-Z_1&Z_1 \end{pmatrix} + \frac{1}{\alpha} \begin{pmatrix} I_d&-I_d\\-I_d&I_d \end{pmatrix}}_{=:J_\alpha'} + \underbrace{ L\begin{pmatrix} 0&Z_2\\Z_2^t&Z_3 \end{pmatrix}}_{=:J_\alpha''}
\end{equation}
and
\begin{multline*}
Z_1:=I_d - \frac{x_\alpha-y_\alpha}{|x_\alpha-y_\alpha|} \otimes \frac{x_\alpha-y_\alpha}{|x_\alpha-y_\alpha|}, \quad Z_2:=D\phi(y_\alpha)\otimes \frac{x_\alpha-y_\alpha}{|x_\alpha-y_\alpha|}, \\
 \mbox{and} \quad Z_3:=-(Z_2+Z_2^t)+D^2\phi(y_\alpha)|x_\alpha-y_\alpha|.
\end{multline*}
With $s>0$, we multiply both sides of \eqref{e.mats-g} on the right by the matrix
\begin{equation}\label{}
A_s:= \frac12\begin{pmatrix} s^2  \sigma^t(x_\alpha) \sigma(x_\alpha) & s \sigma^t(x_\alpha) \sigma(y_\alpha)  \\ s \sigma^t(y_\alpha) \sigma(x_\alpha)  &\sigma^t(y_\alpha) \sigma(y_\alpha) \end{pmatrix} = \frac12\begin{pmatrix}  s \sigma(x_\alpha) \\ \sigma(y_\alpha) \end{pmatrix}^t \begin{pmatrix}  s \sigma(x_\alpha) \\ \sigma(y_\alpha) \end{pmatrix} \geq 0
\end{equation}
and then take the trace of the result to obtain
\begin{equation}\label{e.pause-g}
\tr\left( sA(x_\alpha) X_{\ep,\alpha} - A(y_\alpha) Y_{\ep,\alpha} \right) \leq \tr\left( J_\alpha A_s \right) + \ep \tr\left( J_\alpha^2 A_s \right). 
\end{equation}

The rest of the argument is concerned with deriving contradiction from~\eqref{e.PDEs-g},~\eqref{e.PDEs-g2} and~\eqref{e.pause-g}.

\emph{Step 3.} 
We estimate the right side of~\eqref{e.pause-g} from above, ignoring the second term. In a very similar way to the first string of inequalities in Step~0, we get
\begin{equation}\label{}
\tr\left( J_\alpha' A_s \right) \leq \Lambda_2^2 \left( \frac{L\phi(y_\alpha)}{|x_\alpha-y_\alpha|}+  \frac{1}{\alpha}\right)  \left( (s-1)^2 + |x_\alpha - y_\alpha |^2  \right)
\end{equation}
and, by a routine calculation,
\begin{equation*}\label{}
\tr \left( J_\alpha'' A_s \right) \leq L\Lambda_2^2\left( (s-1)+|x_\alpha-y_\alpha| \right) |D\phi(y_\alpha)|   + \frac12L\Lambda_2^2  |D^2 \phi(y_\alpha)| \cdot |x_\alpha-y_\alpha|.
\end{equation*}
We set $s:=1+\beta|x_\alpha-y_\alpha|$, with $\beta>0$  selected below, sum the previous two lines and express some quantities in terms of $|P_\alpha| =L\phi(y_\alpha)+ \alpha^{-1} |x_\alpha-y_\alpha|$ and $|Q_\alpha| = L|x_\alpha-y_\alpha|\cdot |D\phi(y_\alpha)|$ to obtain
\begin{multline} \label{e.uppbndp}
\tr\left( J_\alpha A_s\right) \\ \leq \Lambda_2^2 |P_\alpha| \left( \beta^2+1 \right) \left|x_\alpha-y_\alpha\right|  + \Lambda_2^2 |Q_\alpha|\left( \beta+1 \right) + \frac12L\Lambda_2^2 |D^2\phi(y_\alpha)|\cdot |x_\alpha-y_\alpha|.
\end{multline}

\emph{Step 4.}
We estimate the left side of~\eqref{e.pause-g} from below. We proceed in a similar way as in Step~0, using the inequalities~\eqref{e.PDEs-g} and~\eqref{e.PDEs-g2} with the structural conditions~\eqref{e.Hsubq} and~\eqref{e.HsubqLip}; unlike Step~1, here we also need~\eqref{e.HsubqDp}. In preparation to apply the latter, and for future reference, we first record some estimates involving the quantities~$|Q_\alpha|$ and~$|P_\alpha|$. By~\eqref{e.phi}, we have
\begin{equation} \label{e.PQdp}
\frac{|Q_\alpha|}{|x_\alpha-y_\alpha|} \leq C \phi^{m-1}(y_\alpha) \left| P_\alpha \right| \leq C L^{1-m} \left| P_\alpha \right|^{m}.
\end{equation}
In the subquadratic case that $1< m\leq 2$, the first inequality of~\eqref{e.PQdp} and~\eqref{e.capt} yield
\begin{equation} \label{e.PQd}
\lim_{\alpha\to 0} \frac{|Q_\alpha|}{|P_\alpha|}\leq \lim_{\alpha\to 0} C |\phi(y_\alpha)|^{\m-1}|x_\alpha-y_\alpha|=0.
\end{equation}
In the superquadratic case $m>2$, we use the first inequality of~\eqref{e.PQdp} and~\eqref{e.PaQa} to get 
\begin{equation} \label{e.buggy}
|Q_\alpha| \leq C L^{1-m} \tilde K^{m-1} |P_\alpha|. 
\end{equation}
By imposing the restriction $L\geq C\tilde K$, we may assume  in both cases that $|Q_\alpha| \leq |P_\alpha|$ for small~$\alpha$. So we henceforth assume
\begin{equation} \label{e.restr1}
L\geq C \left( \left( \frac{\Lambda_2^2}{a_2}\right)^{\frac1{\m-1}} + \left( \frac {M_2}{a_2}\right)^{\frac1\m}  \right).
\end{equation}
For sufficiently small $\alpha$, we now estimate
\begin{align*} \label{}
\lefteqn{ \tr\left( sA(x_\alpha) X_{\ep,\alpha} - A(y_\alpha) Y_{\ep,\alpha} \right)} \qquad & \\ 
& \geq s H\left(P_\alpha,x_\alpha\right) - H\left(P_\alpha-Q_\alpha,y_\alpha\right) 
+\delta (s u(x_\alpha)-u(y_\alpha)) \\
& \geq(s-1)H\left(P_\alpha,x_\alpha\right)  - \Lambda_1 \left( 1+2\left| P_\alpha\right| \right)^{\m-1} \left|Q_\alpha\right| - \left( \Lambda_1\left| P_\alpha\right|^{\m}+M_2 \right) |x_\alpha-y_\alpha|\\
& \qquad +\delta (s-1) u(x_\alpha) \\ 
& \geq (s-1)\left( a_2\left| P_\alpha\right|^\m  -M_2 \right) - \Lambda_1 \left( 1+2\left| P_\alpha\right| \right)^{\m-1} \left|Q_\alpha\right| - \left( \Lambda_1\left| P_\alpha\right|^{\m}+M_2 \right) |x_\alpha-y_\alpha|   \\
& \qquad +\delta (s-1) u(x_\alpha).
\end{align*}
Here we used~\eqref{e.PDEs-g} and~\eqref{e.PDEs-g2} in second line, ~\eqref{e.HsubqLip},~\eqref{e.HsubqDp} and~$|Q_\alpha| \leq |P_\alpha|$ to get the third line, and finally~\eqref{e.Hsubq} in the fourth line.

\emph{Step 5.}
We complete the proof by combining the last inequality of Step~4 with~\eqref{e.pause-g} and~\eqref{e.uppbndp}. We also insert $s:=1+\beta|x_\alpha-y_\alpha|$ for $\beta > 1$ to be selected and then send $\ep \to 0$, to obtain
\begin{align*} \label{}
\beta \left( a_2 \left| P_\alpha\right|^\m - M_2  \right) \left|x_\alpha-y_\alpha\right| & \leq \Lambda_1 \left( 1+2\left| P_\alpha\right| \right)^{\m-1} \left|Q_\alpha\right| + \left( \Lambda_1\left| P_\alpha\right|^{\m}+M_2 \right) |x_\alpha-y_\alpha| \\
& \quad +\Lambda_2^2 |P_\alpha| \left( \beta^2+1 \right) |x_\alpha-y_\alpha|+ \Lambda_2^2 |Q_\alpha|\left( \beta+1\right) \\
& \quad  + \frac12 L\Lambda_2^2 |D^2\phi(y_\alpha)| |x_\alpha-y_\alpha| - \beta \delta u(x_\alpha) |x_\alpha-y_\alpha|.
\end{align*}
Dividing by $|x_\alpha-y_\alpha|$ and rearranging, using~\eqref{e.PQdp} to estimate the terms with $|Q_\alpha|$ and the fact that $|P_\alpha| \geq L \geq 1$ to simplify $1+2|P_\alpha|\leq 3|P_\alpha|$, we arrive at
\begin{align*} \label{}
\left(\beta a_2 -\Lambda_1 \right) \left| P_\alpha \right|^\m & \leq
M_2(\beta +1)+ C\Lambda_1 \phi^{m-1}(y_\alpha)\left| P_\alpha \right|^{\m} +\Lambda_2^2 |P_\alpha|\left( \beta^2+1 \right)  \\
& \quad + C\Lambda_2^2 L^{1-m} \left| P_\alpha \right|^m \left( \beta+1 \right)   + \frac12L\Lambda_2^2 |D^2\phi(y_\alpha)|+\delta \beta \|u\|_{L^\infty(B_{3/2})} .
\end{align*}
Here and in the rest of the argument, $C>0$ may depend on $d$ and $m$. To estimate the term involving $D^2\phi(y_\alpha)$, we use~\eqref{e.phi} which gives
\begin{equation*} \label{}
 L \left|D^2\phi(y_\alpha) \right|  \leq CL^{1-\m} \phi^{\m-1}(y_\alpha) |P_\alpha|^\m.
\end{equation*}
Using this and rearranging, we obtain
\begin{multline*} \label{}
\left(\beta a_2 - C \Lambda_1 \phi^{\m-1}(y_\alpha) - C\Lambda_2^2 L^{1-\m}(\beta+\phi^{m-1}(y_\alpha) ) \right) \left| P_\alpha \right|^\m  \\
  \leq M_2(\beta +1) + \Lambda_2^2|P_\alpha| (\beta^2+1) +\beta \delta\|u\|_{L^\infty(B_{3/2})}.
\end{multline*}
We now impose a second restriction on $L$, namely that $L\geq \left( C\Lambda_2^2 / a_2 \right)^{1/(m-1)}$, which allows us to simplify the term in parentheses on the left side of the last inequality, to get
\begin{equation*} \label{}
\left(\frac12 \beta a_2-C (1+\Lambda_1) \phi^{\m-1}(y_\alpha) \right) \left| P_\alpha \right|^\m  \leq M_2(\beta +1) + \Lambda_2^2|P_\alpha| (\beta^2+1)+ \beta \delta\|u\|_{L^\infty(B_{3/2})}.
\end{equation*}
Next, we make our choice of $\beta$: we take
\begin{equation*} \label{}
\beta:= \frac {C}{a_2} \left( 1+ \Lambda_1\right) \phi^{\m-1}(y_\alpha),
\end{equation*}
which leads to the estimate 
\begin{equation*} \label{}
(1+\Lambda_1) \phi^{\m-1}(y_\alpha) \left| P_\alpha \right|^\m \leq \beta M_2+ \beta^2\Lambda_2^2|P_\alpha|+ \beta \delta\|u\|_{L^\infty(B_{3/2})}.
\end{equation*}
Dividing both sides by $(1+\Lambda_1) \phi^{\m-1}(y_\alpha)$, we get
\begin{equation*} \label{}
\left| P_\alpha \right|^\m \leq \frac{C}{a_2} \left( M_2 + \beta \Lambda_2^2\left|P_\alpha\right| + \delta \| u\|_{L^\infty(B_{3/2})} \right).
\end{equation*}
Using~\eqref{e.PQdp} to estimate the middle term on the right side, we obtain 
\begin{equation*} \label{}
\left| P_\alpha \right|^\m - \frac{C\Lambda_2^2(1+\Lambda_1)}{a_2^2} L^{1-m} \left| P_\alpha \right|^m \leq \frac{C}{a_2} \left( M_2 + \delta \| u\|_{L^\infty(B_{3/2})} \right).
\end{equation*}
By strengthening the second restriction to $L\geq \left(C \Lambda_2^2(1+\Lambda_1)/a_2^2 \right)^{1/(m-1)}$, we obtain
\begin{equation*} \label{}
L^m \leq \left| P_\alpha \right|^m \leq \frac{C}{a_2} \left( M_2 + \delta \| u\|_{L^\infty(B_{3/2})} \right). 
\end{equation*}
In conclusion, recalling also~\eqref{e.restr1}, we obtain a contradiction unless 
\begin{equation*} \label{}
L \leq C \left\{  \left( \frac{(1+\Lambda_1)\Lambda_2^2}{a_2^2} \right)^{1/(m-1)} + \left( \frac{M_2 + \delta \| u\|_{L^\infty(B_{3/2})} }{a_2}\right)^{1/m} \right\},
\end{equation*}
for a large enough $C> 0$ depending only on $\d$ and $m$. 
\end{proof}

\begin{remark}
It is possible to extend the argument above to the case in which ~$H$ is merely superlinear in $p$. One obtains explicit Lipschitz estimates for solutions which depend in an appropriate way on the rate of superlinear growth of $H$ in $p$.
\end{remark}

\subsection{Interior Lipschitz estimates: The time-dependent case.}

In this subsection we prove Lipschitz estimates for solutions of the time-dependent equation
\begin{equation}\label{e-C}
u_t-\tr\left(A(x)D^2u\right)+H(Du,x)=0 \quad \text{in} \ \Rd \times (0,\infty).
\end{equation}
The extra difficulty is dealing with the time variable, since the equation is not coercive in $t$. 

We denote $Q_r:=B_r \times (0,r)$.

\begin{lem}\label{T.squad}
Let $T_0>0$  and assume that $m>2$ and $u \in \USC(Q_2)$ satisfy
\begin{equation}
u_t - \Lambda_2^2 |D^2 u| + a_2 |Du|^m \leq M_2 \quad \text{in}\ Q_2,
\end{equation}
and
\begin{equation}\label{ut-bdd-below}
u_t \geq -T_0 \quad \text{in} \ Q_2.
\end{equation}
Then for every $x,y \in B_1$ and $t \in [0,2]$,
\begin{equation}
|u(x,t)-u(y,t)| \leq K |x-y|^\gamma, \quad \text{for}\ \gamma:=\frac{m-2}{m-1},
\end{equation}
and the constant $K$ is given by
\begin{equation*}
K:=C \left( \left(\frac{\Lambda_2^2}{a_2}\right)^{\frac{1}{m-1}}+\left(\frac{M_2+T_0}{a_2} \right)^{\frac{1}{m}} \right)
\end{equation*}
for $C=C(d,m)$.
\end{lem}

\begin{proof}
In light of \eqref{ut-bdd-below}, it is straightforward to check that for each $t \in [0,2]$, $u(\cdot,t)$ is a subsolution of
\begin{equation*}
-\Lambda_2^2 |D^2 u| + a_2 |Du|^m \leq M_2 + T_0 \quad \text{in} \ B_2.
\end{equation*}
The result therefore follows immediately from~Lemma \ref{l.superquad}. 
\end{proof}

\begin{prop}\label{C.LIP}
Suppose that  $u \in C(\overline Q_2)$  is a solution of \eqref{e-C} with $u(\cdot,0)=u_0 \in C^{0,1}(\overline B_2)$ and assume that there exists a constant $T_0>0$ such that
\begin{equation}\label{e.c.lip-t}
u_t \geq - T_0 \quad \text{in}\ Q_2.
\end{equation}
Then, for all $(x,t),(y,t) \in B_1 \times (0,2)$,
\begin{equation}\label{e.c.lip}
|u(x,t)-u(y,t)| \le K |x-y|,
\end{equation}
for $K>0$ given by
\begin{equation*} \label{}
K :=  C \left\{ \left( \frac{(1+\Lambda_1)^{1/2}\Lambda_2}{a_2} \right)^{2/(\m-1)}+\left( \frac{M_2 + T_0 }{a_2}\right)^{1/\m}  \right\}
\end{equation*}
where $C>0$ depends only on $d$ and $\m$.
\end{prop}

\begin{proof} 

We present the proof in several steps.

{\it Step 1.}
We take $L>0$ and~$(\hat x,\hat t) \in B_1 \times [0,2]$ such that 
\begin{equation}\label{e.C.liploc}
\limsup_{x \to\hat x} \frac{ u(\hat x,\hat t) - u(x,\hat t)}{|\hat x-x|} > L 
\end{equation}
and show that $L>0$ cannot be too large. Choose $\phi :  B_{3/2} \to [1,\infty)$ to be the same cutoff function as in the proof of Theorem~\ref{t.LIP}. Recall that $\phi$ satisfies $\phi \equiv 1$ on $B_1$, $\phi(x) \to +\infty$ as $|x| \to \partial  B_{3/2}$, and  for each $x\in  B_{3/2}$,
\begin{equation*}
\left|D\phi(x) \right| \leq C \left( \phi(x) \right)^\m \quad \mbox{and} \quad \left| D^2\phi(x) \right| \leq C \left( \phi(x) \right)^{2\m -1}.
\end{equation*}
For each $\alpha >0$ sufficiently small, by the same argument in Step 1 of the proof of Theorem \ref{t.LIP}, there exist points $(x_\alpha,t_\alpha), (y_\alpha,t_\alpha)\in  B_{3/2} \times [0,2]$ which satisfy
\begin{multline}\label{e.C.lip.del}
u(x_\alpha,t_\alpha) - u(y_\alpha,t_\alpha) - L\phi(y_\alpha) |x_\alpha-y_\alpha| - \frac1{2\alpha} |x_\alpha-y_\alpha|^2 \\
= \sup_{(x,t),(y,t)\in  B_{3/2}\times [0,2]} \left( u(x,t)-u(y,t)-L\phi(y) |x-y| - \frac1{2\alpha} |x-y|^2\right) > 0.
\end{multline}
Returning to~\eqref{e.C.lip.del} and using the continuity of $u$ and the bound $|x_\alpha-y_\alpha| \leq C\delta^{1/2}$, we get
\begin{multline}\label{e.C.bdd.del}
L\phi(y_\alpha) |x_\alpha-y_\alpha| +  \frac1{2\alpha} |x_\alpha-y_\alpha|^2\\
 \leq \sup \left\{ u(y) - u(z) \, : \, y,z\in B_{3/2},  \ |y-z| \leq C\delta^{1/2}  \right\}  \longrightarrow 0.
\end{multline}
When $m>2$, we actually have a  better estimate, in light of Lemma \ref{T.squad},
\begin{equation}\label{e.C.bdd.more}
L\phi(y_\alpha) |x_\alpha-y_\alpha| +  \frac1{2\alpha} |x_\alpha-y_\alpha|^2 \leq u(x_\alpha,t_\alpha)-u(y_\alpha,t_\alpha) \leq
\widetilde K |x_\alpha-y_\alpha|^\gamma,
\end{equation}
where $\gamma=(m-2)/(m-1)$ and $\widetilde K$ is the constant $K$ given in Lemma \ref{T.squad}. The above implies that
\begin{equation}\label{ad-f-t1}
(\phi(y_\alpha))^{m-1} |x_\alpha-y_\alpha| \leq L^{1-m} \widetilde K^{m-1}.
\end{equation}
An extra difficulty not arising in the stationary case involves the time variable $t_\alpha$. In particular, we need to handle the cases $t_\alpha=0$ and $t_\alpha=2$. By choosing $L>\|Du_0\|_{L^\infty(B_2)}$, we have $t_\alpha>0$, which excludes the first case. For the second case, we add a penalized term as follows. For each $\lambda>0$ sufficiently small,  there exist points $(x_\alpha^\lambda,t_\alpha^\lambda), (y_\alpha^\lambda,t_\alpha^\lambda)\in  B_{3/2} \times (0,2)$ which satisfy
\begin{multline}\label{e.C.del-al}
u(x_\alpha^\lambda,t_\alpha^\lambda) - u(y_\alpha^\lambda,t_\alpha^\lambda) - L\phi(y_\alpha^\lambda) |x_\alpha^\lambda-y_\alpha^\lambda| - \frac1{2\alpha} |x_\alpha^\lambda-y_\alpha^\lambda|^2-\frac{\lambda}{2-t_\alpha^\lambda} \\
= \sup_{(x,t),(y,t)\in  B_{3/2}\times [0,2]} \left( u(x,t)-u(y,t)-L\phi(y) |x-y| - \frac1{2\alpha} |x-y|^2-\frac{\lambda}{2-t}\right) > 0.
\end{multline}
Note that $(x_\alpha^\lambda,y_\alpha^\lambda,t_\alpha^\lambda) \to (x_\alpha,y_\alpha,t_\alpha)$ as $\lambda \to 0$. We often drop the superscripts and write $(x_\alpha,y_\alpha,t_\alpha) = (x_\alpha^\lambda,y_\alpha^\lambda,t_\alpha^\lambda)$  for simplicity, if there is no confusion.

{\it Step 2.}
Applying \cite[Theorem 8.3]{CIL}, we obtain, for each $\ep > 0$ and sufficiently small $\alpha, \lambda > 0$, a number $\tau \in \R$ and symmetric matrices $X_{\ep,\alpha}, Y_{\ep,\alpha}\in \Sy$ such that 
\begin{equation}\label{e.C.mat}
\begin{pmatrix} X_{\ep,\alpha} &0\\ 0&-Y_{\ep,\alpha} \end{pmatrix} \leq J_\alpha + \ep J_\alpha^2, 
\end{equation}
as well as
\begin{multline}\label{e.C.PDE}
\tau+\frac{\lambda}{(2-t_\alpha)^2}-\tr\left( A(x_\alpha) X_{\ep,\alpha} \right) + H\left( \left( L \phi(y_\alpha) + \frac{|x_\alpha-y_\alpha|}{\alpha} \right) \frac{x_\alpha-y_\alpha}{|x_\alpha-y_\alpha|} , x_\alpha \right) \leq 0 \\ 
 \leq \tau -\tr\left( A(y_\alpha) Y_{\ep,\alpha} \right) + H\bigg( \underbrace{\left( L \phi(y_\alpha) + \frac{|x_\alpha-y_\alpha|}{\alpha} \right) \frac{x_\alpha-y_\alpha}{|x_\alpha-y_\alpha|}}_{=:P_\alpha} - \underbrace{L|x_\alpha-y_\alpha| D\phi(y_\alpha)}_{=:Q_\alpha} , y_\alpha \bigg),
\end{multline}
where we have defined
\begin{equation}\label{}
J_\alpha:=\underbrace{\frac{L\phi(y_\alpha)}{|x_\alpha-y_\alpha|} \begin{pmatrix} Z_1&-Z_1\\-Z_1&Z_1 \end{pmatrix} + \frac{1}{\alpha} \begin{pmatrix} I_d&-I_d\\-I_d&I_d \end{pmatrix}}_{=:J_\alpha'} + \underbrace{ L\begin{pmatrix} 0&Z_2\\Z_2&Z_3 \end{pmatrix}}_{=:J_\alpha''}
\end{equation}
and
\begin{multline*}
Z_1:=I_d - \frac{x_\alpha-y_\alpha}{|x_\alpha-y_\alpha|} \otimes \frac{x_\alpha-y_\alpha}{|x_\alpha-y_\alpha|}, \qquad Z_2:=D\phi(y_\alpha)\otimes \frac{x_\alpha-y_\alpha}{|x_\alpha-y_\alpha|},
\\ \mbox{and} \qquad Z_3:=-2Z_2+D^2\phi(y_\alpha)|x_\alpha-y_\alpha|.
\end{multline*}
By \eqref{e.c.lip-t},
\begin{equation}\label{e.C.tau}
\tau \geq -T_0.
\end{equation}
With $s>0$, we multiply both sides of \eqref{e.mats-g} on the right by the matrix
\begin{equation}\label{}
A_s:= \begin{pmatrix} s^2  \sigma(x_\alpha) \sigma^t(x_\alpha) & s \sigma(x_\alpha) \sigma^t(y_\alpha)  \\ s \sigma(y_\alpha) \sigma^t(x_\alpha)  &\sigma(y_\alpha) \sigma^t(y_\alpha) \end{pmatrix} = \begin{pmatrix}  s \sigma(x_\alpha) \\ \sigma(y_\alpha) \end{pmatrix} \begin{pmatrix}  s \sigma(x_\alpha) \\ \sigma(y_\alpha) \end{pmatrix}^t \geq 0
\end{equation}
and then take the trace of the result to obtain
\begin{equation}\label{e.C.trace}
\tr\left( sA(x_\alpha) X_{\ep,\alpha} - A(y_\alpha) Y_{\ep,\alpha} \right) \leq \tr\left( J_\alpha A_s \right) + \ep \tr\left( J_\alpha^2 A_s \right). 
\end{equation}

{\it Step 3.}
We set $s:=1+\beta|x_\alpha-y_\alpha|$, with $\beta>0$  selected below, and obtain the following estimate for $\tr\left( J_\alpha A_s \right)$:
\begin{equation} \label{e.C.trace.up}
\tr\left( J_\alpha A_s\right) \leq 
\Lambda_2^2|P_\alpha| \left( \beta^2+1 \right) |x_\alpha-y_\alpha| + \Lambda_2^2 |Q_\alpha|\left( \beta+1 \right) + \frac{L\Lambda_2^2}{2} |D^2\phi(y_\alpha)| |x_\alpha-y_\alpha|.
\end{equation}
To bound the left side of~\eqref{e.C.trace} from below, we use the inequalities in \eqref{e.C.PDE} with the structural conditions~\eqref{e.Hsubq}, \eqref{e.HsubqLip}, and \eqref{e.HsubqDp}. Following the same computation as in Step 4 of the proof of Theorem \ref{t.LIP}, we could assume that $|Q_\alpha| \leq |P_\alpha|$ provided that
\begin{equation}\label{ad-f-t2}
L \geq C \left (  \left(\frac{\Lambda_2^2}{a_2}\right)^{\frac{1}{m-1}}+\left( \frac{M_2+T_0}{a_2}\right)^{\frac{1}{m}} \right).
\end{equation}
For sufficiently small $\alpha, \lambda$, we get,
\begin{align*} \label{}
\lefteqn{ \tr\left( sA(x_\alpha) X_{\ep,\alpha} - A(y_\alpha) Y_{\ep,\alpha} \right)} \qquad & \\
& \geq s H\left(P_\alpha,x_\alpha\right) - H\left(P_\alpha-Q_\alpha,y_\alpha\right) + (s-1) \tau+s\frac{\lambda}{(2-t_\alpha)^2}  \\
& \geq (s-1) \left( H\left(P_\alpha,x_\alpha \right) +\tau \right) - \Lambda_1 \left( 1+2\left| P_\alpha \right| \right)^{\m-1} \left|Q_\alpha\right| - \left( \Lambda_1\left| P_\alpha \right|^{\m}+M_2 \right) |x_\alpha-y_\alpha| \\ 
& \geq (s-1)\left( a_2\left| P_\alpha \right|^\m  -M_2+\tau \right) - \Lambda_1 \left( 1+2\left| P_\alpha \right| \right)^{\m-1} \left|Q_\alpha\right| - \left( \Lambda_1\left| P_\alpha \right|^{\m}+M_2 \right) |x_\alpha-y_\alpha|,
\end{align*}
where the term $s\lambda (2-t_\alpha)^{-2}$ was ignored because of its sign.
Next we combine the above with~\eqref{e.C.trace} and~\eqref{e.C.trace.up}, inserting $s:=1+\beta|x_\alpha-y_\alpha|$ and sending $\alpha,\ep \to 0$, to obtain
\begin{align*} \label{}
\beta \left( a_2 \left| P_\alpha \right|^\m - M_2+\tau \right) \left|x_\alpha-y_\alpha \right| & \leq \Lambda_1 \left( 1+2\left| P_\alpha \right| \right)^{\m-1} \left|Q_\alpha\right| + \left( \Lambda_1\left| P_\alpha \right|^{\m}+M_2 \right) |x_\alpha-y_\alpha| \\ 
& \quad +\Lambda_2^2 |P_\alpha| \left( \beta^2+1 \right) |x_\alpha-y_\alpha| + \Lambda_2^2|Q_\alpha|\left( \beta +1 \right) \\ 
&\quad+\frac{1}{2} L\Lambda_2^2 |D^2\phi(y_\alpha)|\cdot |x_\alpha-y_\alpha|.
\end{align*}
Repeating the computations in Step 5 in the proof of Theorem \ref{t.LIP}, we obtain finally that
\begin{equation*}
L \leq C \left\{ \left(\frac{(1+\Lambda_1)\Lambda_2^2}{a_2^2}\right)^{\frac{1}{m-1}}+\left(\frac{M_2+T_0}{a_2} \right)^{\frac{1}{m}} \right\}.\qedhere
\end{equation*}
\end{proof}

%

\section{Boundary value problem with state constraints} \label{BVSC}

Motivated by problems in stochastic optimal control, Lasry and Lions~\cite{LL} initiated the study of boundary-value problems for viscous Hamilton-Jacobi equations with state-constraints~25 years ago. The topic has since attracted the attention of many researchers, although most of the studies to date apply only to superquadratic equations or those in which the diffusion matrix~$A$ is isotropic and uniformly elliptic (that is, up to an affine change of variables, the second-order term is the Laplacian). The reason is that (as shown in~\cite{LL}) for a constant, uniformly elliptic diffusion matrix, we can determine, by explicit computation, a precise blow-up rate for solutions near the boundary of the domain-- which then allows for the application of the comparison principle. Meanwhile, the superquadratic case is quite easy to analyze in view of Lemma~\ref{l.superquad}.

Here we present, as far as we are aware, the first well-posedness results for state constrained problems for viscous Hamilton-Jacobi equations with nonisotropic diffusions and subquadratic Hamiltonians. The problem we study has the following form:
\begin{equation}\label{e.state}
\left\{ \begin{aligned}
 & \delta u -\tr\left( A(x) D^2u \right) + H(Du,x) \leq f &  \mbox{in} & \  U,\\
 & \delta u -\tr\left( A(x) D^2u \right) + H(Du,x) \geq f &  \mbox{on} & \ \overline U,
\end{aligned} \right.
\end{equation}
where~$\delta>0$ and $U \subseteq \Rd$ is a given domain.

Let us recall the precise interpretation of~\eqref{e.state}.

\begin{definition}
\label{d.sc}
We say that $u\in \USC(U)$ is \emph{subsolution} of~\eqref{e.state} if $u$ is a viscosity solution of the inequality
\begin{equation*} \label{}
\delta u -\tr\left( A(x) D^2u \right) + H(Du,x) \leq f(x) \quad  \mbox{in}  \  U.
\end{equation*}
We say that $v\in \LSC(\overline U)$ is \emph{supersolution} of~\eqref{e.state} if, for every smooth function $\varphi\in C^2(\overline U )$ and point $x_0 \in \overline U$ such that 
\begin{equation*} \label{}
\min_{ \overline U } (v-\varphi) = (v-\varphi)(x_0),
\end{equation*}
we have
\begin{equation*} \label{}
\delta v(x_0) -\tr\left( A(x_0) D^2\varphi(x_0) \right) + H(D\varphi(x_0),x_0) \geq f(x_0).
\end{equation*}
Of course, we say that $u\in C(U) \cap \LSC(\overline U )$ is a \emph{solution} of~\eqref{e.state} if it is both a subsolution and supersolution~\eqref{e.state}. We remark that we allow functions in $\LSC(\overline U)$ to take values in $\R \cup\{ +\infty\}$. In particular, a solution of~\eqref{e.state} may be $+\infty$ on $\partial U$. 
Note that if $u(x_0)=+\infty$ for some $x_0 \in \partial U$, then $u$ is automatically a supersolution of \eqref{e.state} at $x_0$.
\end{definition}

The following theorem is our main result concerning the state constrained problem~\eqref{e.state}. It may be compared with~\cite[Theorem~I.1]{LL}. See also Barles and Da Lio~\cite{BF} for results on state constraints problems as well as problems with more general boundary conditions.
In case $A\equiv 0$, we refer the readers to Fathi and Siconolfi \cite{FaSi} and Mitake~\cite{Mi}.

\begin{thm}
\label{t.state.q}
Let $U\subseteq \Rd$ be open and connected and $f\in C^{0,1}_{\mathrm{loc}}(U) \cap L^\infty(U)$. 
Define the maximal subsolution of~\eqref{e.state} by
\begin{equation}\label{e.state.q.fmla}
u(x) := \sup\left\{ w(x) \, :\, w\in \USC(\overline U) \ \mbox{is a subsolution of~\eqref{e.state} in $U$} \right\}.
\end{equation}
Then~$u\in C^{0,1}_{\mathrm{loc}}(U) \cap \LSC(\overline U)$ and~$u$ is a solution of~\eqref{e.state}.
\end{thm}

The previous result states in particular that~\eqref{e.state} has a unique maximal solution. In certain cases (e.g., if $U=\Rd$ and $H$ is uniformly coercive, or $\partial U$ is smooth and $A$ is well-behaved) we can prove a complete well-posedness result, showing that~\eqref{e.state.q.fmla} is the unique bounded-below solution of~\eqref{e.sc.bvp}. In general, we do not know how to prove such a uniqueness result. However, obtaining such a uniqueness statement is in many situations secondary to simply showing that the function in~\eqref{e.state.q.fmla} is continuous-- which is already a kind of well-posedness result. Indeed, in optimal control theory it can often be shown that $u$ is the value function of the problem by maximality (and therefore the primary object of interest).

\medskip

We continue by introducing notations and making some preliminary observations. Essentially all the difficulty in proving Theorem~\ref{t.state.q} lies in is handing the case that~$1 < m \leq 2$ and~$U$ is bounded and smooth. For this discussion, we proceed under these assumptions. For~$\ep>0$, we denote
\begin{equation*}\label{}
U_\ep:= \left\{ x\in U \,:\, \dist(x,\partial U) > \ep \right\} \qquad \mbox{and} \qquad
U^\ep:= \left\{ x\in \Rd \,:\, \dist(x,U) < \ep\right\}.
\end{equation*}
Since~$\partial U$ is smooth, we may select a nonnegative function~$\db_U \in C^2( \overline U )$ such that~$0<\db_U \leq 1$ in~$U$ and, for~$0<\ep_0<1$ depending on the geometry of~$U$, we have $\db_U \geq \ep_0$ in $\overline U_{\ep_0}$ and
\begin{equation*}\label{}
\db_U (x) = \dist(x,\partial U) := \inf_{y\in \partial U} |x-y| \qquad \mbox{for every} \ x\in \overline U \setminus U_{\ep_0}.
\end{equation*}
Note that this implies that $|D\db_U| \equiv 1$ in $U \setminus\overline U_{\ep_0}$. Following~\cite{LL}, we introduce
\begin{equation}\label{e.sc.phiU}
\zeta_U(x):= \begin{cases} 
\db_U(x)^{\frac{\m-2}{\m-1}} & 1< \m <2, \\
 1 - \log\left( \db_U(x)\right)  & \m = 2.
\end{cases}
\end{equation}
If the underlying set $U$ can be inferred from the context, we simply write $\db := \db_U$ and $\zeta:=\zeta_U$. Note that $\zeta_U > 1$ in $U$ and $\zeta_U \leq C(m,\ep_0)$ in $\overline U_{\ep_0}$. 

The utility of the test function $\zeta_U$ is due to the scaling of the equation: by a routine computation, we have 
\begin{equation}\label{e.zetaUpro}
| D\zeta_U(x)|^\m  \simeq (\db_U(x))^{-\frac\m{\m-1}} \simeq |D^2\zeta_U(x)|\quad \mbox{in} \ U \setminus U_{\ep_0},
\end{equation}
where the constant of proportionality implicit in the second relation depends on an upper bound for the curvature of $\partial U$. Due to the superlinearity of the gradient term, this means that, for large $C>0$, the function $C\zeta_U$ is a supersolution of~\eqref{e.state} which blows up near $\partial U$, while $c\zeta_U$ will be a subsolution of~\eqref{e.state} for small $c>0$ provided that the diffusion is nondegenerate near $\partial U$ or $f \gtrsim (\db_U(x))^{-\frac\m{\m-1}}$ near $\partial U$. The precise formulation is contained in the following two lemmas.

\begin{lem}\label{l.dvdobs}
Assume that $1< m\leq 2$ and $U$ is a smooth bounded domain. Suppose also that $\delta, a>0$ and $\Lambda,M, \eta \geq 0$ and $u\in \USC(U)$ satisfy
\begin{equation*} \label{}
\delta u - \Lambda^2\left| D^2u \right| + a|Du|^m  \leq M + \eta (\db_U(x))^{-\frac\m{\m-1}}  \quad  \mbox{in}  \  U.
\end{equation*}
Then
\begin{equation*} \label{}
u \leq \frac{M}{\delta} + K \zeta_U \quad \mbox{in} \ U,
\end{equation*}
where $K>0$ denotes, for some $C>0$ depending on $d$, $m$ and the geometry of $U$,
\begin{equation*} \label{}
K : = C \left(  \left( \frac{\Lambda^2}{a} \right)^{\frac1{\m-1}} + \left( \frac{\eta}{a} \right)^{\frac1{\m}} \right).
\end{equation*}
\end{lem}
\begin{proof}
By replacing $u$ by $u - M/\delta$, it suffices to consider the case $M=0$. It also suffices to prove the estimate with $\zeta_{U_\ep}$ in place of $\zeta_{U}$, by continuity. Since $u$ is bounded above on $\overline U_\ep$ and $\zeta_{U_\ep} \rightarrow +\infty$ as $x \to \partial U_{\ep}$, for each $K>0$, the function $x\mapsto u(x) - K \zeta_{U_\ep}(x)$ must attain its supremum over $U_{\ep}$ at some point $x_0 \in U_{\ep}$. It suffices to  show that $u(x_0) < K \zeta_{U_\ep}(x_0)$ for $K>0$ as in the statement of the lemma. Suppose on the contrary that $u(x_0) \geq K \zeta_{U_\ep}(x_0) \geq 0$. Then by the definition of viscosity subsolution, we have 
\begin{align*} \label{}
\eta (\db_U(x_0))^{-\frac\m{\m-1}} & \geq \delta u(x_0) - \Lambda^2K \left| D^2\zeta_{U_{\ep}}(x_0) \right| + aK^m\left|D\zeta_{U_{\ep}}(x_0)\right|^m \\
 & \geq -\Lambda^2 K \left| D^2\zeta_{U_{\ep}}(x_0) \right| + aK^m \left|D\zeta_{U_{\ep}}(x_0)\right|^m.
\end{align*}
We may now make $K>0$ large, using~\eqref{e.zetaUpro}, to obtain a contradiction. We find that we need to take $K$ as in the statement of the lemma with $C>0$ large enough, depending on $d$, $m$ and the geometry of $U$.
\end{proof}

\begin{lem}\label{l.dvdpsiops}
Assume that $1< m\leq 2$ and $U$ is a smooth bounded domain. 
Fix $\eta, a>0$ and $\Lambda\geq 0$. Then the function $v:= k\zeta_U - M/\delta$ is a smooth solution of
\begin{equation*} \label{}
\delta v + \Lambda^2\left| D^2v \right| + a |Dv|^m  \leq -M + \eta (\db_U(x))^{-\frac\m{\m-1}} \quad  \mbox{in}  \  U,
\end{equation*}
provided that, for $c>0$ depending on $d$, $m$ and the geometry of $U$, 
\begin{equation*} \label{}
0< k \leq c \left(  \left( \frac{\Lambda^2}{a} \right)^{\frac1{\m-1}} + \left( \frac{\eta}{a} \right)^{\frac1{\m}} \right).
\end{equation*}
\end{lem}
\begin{proof}
We may assume $M=0$. The proof is then an easy exercise using~\eqref{e.zetaUpro}. 
\end{proof}

\begin{remark}\label{e.bounddvd}
Observe that Lemma~\ref{l.dvdobs} yields local upper bounds for solutions $u\in \USC(B_2)$ of the inequality
\begin{equation*} \label{}
\delta u -\tr\left( A(x) D^2u \right) + H(Du,x) \leq 0. 
\end{equation*}
We have, for $C>0$ depending only on $d$ and $m$,
\begin{equation*} \label{}
\sup_{B_1} u \leq \frac{M_2}{\delta} + C \left( \frac{\Lambda_2^2}{a_2} \right)^{\frac1{\m-1}}.
\end{equation*}
\end{remark}

The proof of Theorem~\ref{t.state.q} in the case $\m>2$ is relatively easy, and we postpone it and concentrate first on the more difficult case that $1<\m\leq 2$.

\begin{proof}[{\bf Proof of Theorem~\ref{t.state.q} in the subquadratic case, $1<\m\leq 2$}]

We take $u$ to be defined by \eqref{e.state.q.fmla}. We first prove the result under the additional hypothesis that $U$ is bounded and smooth and, rather than assuming $f$ to be bounded, we take $f$ to be bounded below but satisfy the growth condition
\begin{equation}\label{e.state.fcoer}
f(x) \geq c (\db_U(x))^{-\frac{\m}{\m-1}} \quad \mbox{for all} \ x\in U.
\end{equation}
These assumptions are removed in the final step of the argument. The strategy in the case of~\eqref{e.state.fcoer} is to consider the function
\begin{equation}\label{def.hatmmuU}
\widehat{u}(x):= \inf\left\{ w(x) \, :\, w\in \LSC(\overline U) \ \mbox{is a supersolution of~\eqref{e.state} and}  \  \ \textstyle{\inf_U w} > -\infty \right\}
\end{equation}
and to argue that $u^* \leq \widehat{u}$ and $\widehat{u} \leq u_*$ in $U$, which of course imply that $u = \widehat{u}=u_*=u^*$ and hence $u$ is continuous. The reason that~\eqref{e.state.fcoer} is helpful is because it allows us to show, using Lemma~\ref{l.dvdobs} and~\ref{l.dvdpsiops}, that $u / \widehat u$ stays bounded in $U$. This allows us to implement a comparison argument based on Lemma~\ref{l.convex2}.

{\it Step 1.} We show that $u$ is well-defined, bounded below, locally bounded above and obtain a precise blow up rate near $\partial U$. First we observe that $u \geq -\delta^{-1}\Lambda_1$,
since the right side of this inequality is a (constant) subsolution of~\eqref{e.state}. Moreover, it follows from~Lemma~\ref{l.dvdpsiops} that~$u\geq c \zeta_U - C$ in $U$ for some positive constants~$c,C>0$. To obtain an upper bound, we observe that, following the same calculation as in the proof of Lemma~\ref{l.dvdobs}, using the assumed smoothness of $U$, there exists~$\gamma_0, C> 0$ such that for every~$0<\gamma<\gamma_0$, the function~$C \zeta_{U_\gamma}$ is a strict supersolution of~\eqref{e.state} in~$U_\gamma$. Fix $w\in\USC(U)$ in the admissible class in the definition~\eqref{e.state.q.fmla} of $u$. Since $\zeta_{U_\gamma}$ blows up on $\partial U_\gamma$, there exists $x_0\in U_\gamma$ such that $\sup_{U_\gamma}  (w-C\zeta_{U_\gamma}) = (w-C\zeta_{U_\gamma})(x_0)$. At $x_0$, we find that
\begin{multline*}\label{}
\delta C \zeta_{U_\gamma}(x_0)  -\tr\left( A(x_0) D^2(C\zeta_{U_\gamma})(x_0) \right) + H(D(C\zeta_{U_\gamma})(x_0),x_0) >  f(x_0) \\ \geq \delta w(x_0) -\tr\left( A(x_0) D^2(C\zeta_{U_\gamma})(x_0) \right) + H(D(C\zeta_{U_\gamma})(x_0),x_0) .
\end{multline*}
Rearranging this gives $(w-C\zeta_{U_\gamma})(x_0) < 0$, and thus $w\leq C\zeta_{U_\gamma}$ in $U_\gamma$. Sending $\gamma \to 0$ yields that $w \leq C\zeta_U$ in $U$. This holds for all $w$ inside the supremum on the right side of~\eqref{e.state.q.fmla}, and thus $u\leq C\zeta_U$ in $U$. 

We have shown that
\begin{equation}
\label{e.state.q.gr1}
c \zeta_U - C\leq u  \leq C \zeta_U \quad \mbox{in} \  U.
\end{equation}

{\it Step 2.} We establish an estimates on the blow-up rate of $\widehat{u}$ near $\partial U$. We claim that there exists $c>0$ such that
\begin{equation}
\label{e.state.q.gr2}
\widehat u \geq c \zeta_U -C \quad \mbox{in} \ U.
\end{equation}
By Lemma~\ref{l.dvdpsiops}, for sufficiently small $c>0$ and $\gamma>0$, the function $c\zeta_{U^\gamma}$ is a strict subsolution of~\eqref{e.state} in $U$. Since $\zeta_{U^\gamma}$ is bounded and smooth on~$\overline U$, if $w \in \LSC(\overline U)$ is any function in the admissible class for $\widehat u$, then $w - c\zeta_{U^\gamma}$ must achieve its infimum over $\overline U$ at some point $x_0\in \overline U$. Since $c\zeta_{U^\gamma}-C$ is a strict subsolution and $\widehat u$ is a supersolution of~\eqref{e.state} in $U$, we deduce that 
\begin{multline*}\label{}
\delta (c \zeta_{U^\gamma}(x_0) - C)  -\tr\left( A(x_0) D^2(c\zeta_{U^\gamma})(x_0) \right) + H(D(c\zeta_{U^\gamma})(x_0),x_0) <  f(x_0) \\ \leq \delta w(x_0) -\tr\left( A(x_0) D^2(c\zeta_{U^\gamma})(x_0) \right) + H(D(c\zeta_{U^\gamma})(x_0),x_0) .
\end{multline*}
We deduce that $c\zeta_{U^\gamma}(x_0) - C \leq w(x_0)$. Since $x_0$ is the minimum point of $w - c\zeta_{U^\ep}(x_0)$, we obtain that~$c\zeta_{U^\gamma}(x_0) - C \leq w(x_0)$ in~$U$. Since this holds for all such $w$, we obtain~\eqref{e.state.q.gr2}.

{\it Step 3.} We argue that~$u^*$ is subsolution of~\eqref{e.state} and~$u_*$ is supersolution of~\eqref{e.state}. The first claim is immediate from the definition of $u$ as a supremum of a family of subsolutions and the fact that it is locally bounded in~$U$. This implies in particular that $u^*=u$, so $u \in \USC(U)$. The proof that~$u_*$ is a supersolution of~\eqref{e.state} in $\overline U$ follows the usual Perron method. We give the argument for completeness. Select a smooth function $\phi \in C^{2}(\overline U)$ and a point $x_0\in \overline U$ such that 
\begin{equation}\label{e.state.q.touch}
u_* - \phi \quad \mbox{has a strict local minimum at} \ x_0. 
\end{equation}
We must show that 
\begin{equation}\label{e.state.q.wts}
\delta u_*(x_0) -\tr(A(x_0)D^2\phi(x_0)) + H(D\phi(x_0),x_0) \geq f(x_0). 
\end{equation}
Assuming on the contrary that~\eqref{e.state.q.wts} is false, we use the smoothness of $\phi$ and the definition of $u_*$ to find $r,\theta> 0$ such that 
\begin{equation*}\label{}
\delta u_*(x) -\tr(A(x)D^2\phi(x)) + H(D\phi(x),x) \leq \mu - \theta \quad \mbox{in} \ \overline U \cap B_{r}(x_0).
\end{equation*}
By adding a constant to $\phi$ and shrinking $r>0$, if necessary, we may assume by~\eqref{e.sc.supq.touch} that 
\begin{equation*}\label{}
u_*(x_0) - \phi(x_0) < 0 < u_* - \phi \quad \mbox{on} \ \overline U \cap \partial B_r(x_0).
\end{equation*}
Define the function
\begin{equation*}\label{}
w(x):= \begin{cases}
u(x) & \mbox{in} \ \overline U \setminus  \overline B_r(x_0), \\
\max\{ u(x) , \phi(x) \} & \mbox{in} \ \overline U \cap \overline B_r(x_0).
\end{cases}
\end{equation*}
It is clear from construction that $w$ is a subsolution of~\eqref{e.state} and hence an admissible function in the definition of $u$. Thus $w\leq u$ in $U$. This contradicts the fact that $w(x_0) = \phi(x_0) > u_*(x_0)$ and completes the proof that $u_*$ is a supersolution of~\eqref{e.state} on $\overline U$.

As a consequence of the fact that $u_*$ is a supersolution which is bounded below in $U$, it follows from the definition of~$\widehat u$ that $\widehat{u} \leq u_*$.

{\it Step 4.} We complete the argument under the extra assumption of~\eqref{e.state.fcoer}. According to Lemma~\ref{l.convex2}, for every $\ep > 0$, the function $w:= (1+\ep) \widehat{u} - \ep u$ is a supersolution of~\eqref{e.sc} in $U$. Moreover, by \eqref{e.state.q.gr1} and~\eqref{e.state.q.gr2}, for sufficiently small $\ep > 0$, the function $w$ satisfies $w \geq \frac12 c \zeta_U$ near $\partial U$. In particular, $w$ is bounded below. Therefore, we conclude from the definition of $\widehat{u}$ that $\widehat{u} \leq w$. A rearrangement of this inequality gives $u \leq \widehat u$. Hence $u = \widehat u = u_* = u^*$ and $u$ is continuous in $U$. Since $u\in C(U)$, the Lipschitz estimates from Theorem~\ref{t.LIP} apply and yield that $u \in C^{0,1}_{\mathrm{loc}}(U)$. The fact that~$u$ is the unique bounded-below solution of~\eqref{e.state} is clear from the definitions of $\widehat u$ and $u$ and the fact that $\widehat u = u$.

{\it Step 5.} We present the argument for general bounded $f\in C^{0,1}_{\mathrm{loc}}(U)$. Just as above, we take $u$ to be defined by ~\eqref{e.state.q.fmla} and we show that $u \in C(U)$. We consider the function
\begin{equation*}\label{}
f^\ep (x) : = f(x) + \ep (\db_U(x))^{-\frac\m{\m-1}},
\end{equation*}
and let $u^\ep$ denote the corresponding maximal subsolution with~$f^\ep$ in place of~$f$. It is clear by the obvious monotonicity with respect to $f^\ep$ of the maximal subsolutions that $u \leq u^\ep \leq u^{\ep'}$ in $U$ provided $0< \ep<\ep'\leq 1$. By what we have shown above in Step~1--4, since $f^\ep$ satisfies~\eqref{e.state.fcoer}, we have that~$u^\ep \in C^{0,1}_{\mathrm{loc}}(U)$. Moreover, Theorem~\ref{t.LIP} yields that $\{ u^\ep \}_{0<\ep\leq 1}$ is uniformly Lipschitz in each compact subset of~$U$. Since $u^\ep$ is uniformly from bounded from below $-\delta^{-1}\Lambda_1$ it follows that $u^\ep$ converges locally uniformly in $U$ to a function $v$ as $\ep \to 0$, and $v\in C^{0,1}_{\mathrm{loc}}(U)$. Furthermore, $v \geq u$, by monotonicity, and $v$ is a subsolution of \eqref{e.state}. By the definition of $u$, we have $u \geq v$. Thus $u=v$.

{\it Step 6.} In this final step, we remove the assumption that $U$ is bounded and smooth. We consider instead an increasing  sequence $\{U_k\}_{k\in \N}$ of smooth, bounded domains such that $\cup_{k\in\N} U_k = U$. Denote by $u_k \in C(U_k)$ the corresponding  solution of~\eqref{e.state} in $U_k$, and observe that $\{ u_k \}$ is monotone decreasing, by definition, and equi-Lipschitz in each compact subset of $U$ by what we have shown above. It follows that $u_k \rightarrow v $ locally uniformly for some $v\in C^{0,1}_{\mathrm{loc}}(U)$. We find that $v=u$ by arguing as in Step 5 above. 
\end{proof}

\begin{proof}[{\bf Proof of Theorem~\ref{t.state.q} in the superquadratic case, $\m > 2$}]
According to Remark~\ref{e.bounddvd}, $u$ is locally bounded in $U$. According to Lemma~\ref{l.superquad}, the family of subsolutions in the admissible class in the definition of $u$ is locally equi-continuous in $U$. It follows that $u \in C(U)$ and thus, by Theorem~\ref{t.LIP}, that $u \in C^{0,1}_{\mathrm{loc}}(U)$. The proof that $u$ is a solution of~\eqref{e.state} follows along the lines of Step~3 in the previous argument. 
\end{proof}

\section{The metric problem}
\label{s.mp}

In this section we study the \emph{maximal subsolution} of the equation
\begin{equation} \label{e.sc}
-\tr\left( A(x) D^2u \right) + H(Du,x) = \mu \quad \text{in} \ U \setminus  \overline B_1,
\end{equation}
subject to the constraint $u \leq 0$ on $\overline B_1$.
 Here $U$ is an open subset of $\Rd$ such that 
\begin{equation}\label{e.sc.U}
\overline B_1\subseteq U, \quad U \ \ \mbox{and} \ \ U\setminus \overline B_1 \ \  \mbox{are connected.}
\end{equation}
The maximal subsolution is defined, for all $x\in U$, by
\begin{equation} \label{e.scsol}
m_\mu^U(x):= \sup\left\{ w(x) \,:\, w\in \USC( U ) \ \mbox{is a subsolution of~\eqref{e.sc} and} \ w \leq 0 \ \mbox{on} \ \overline B_1 \right\}.
\end{equation}
The quantity $m_\mu^U$ arises in the theory of optimal stochastic control as it has a natural interpretation as the ``cost of moving a particle from $x$ to $B_1$" for a certain controlled diffusion process (see Remark~\ref{r.control}). As $m_\mu^{\Rd}$ can thus be interpreted as a kind of ``distance,"~\eqref{e.sc} with $U=\Rd$ together with appropriate boundary conditions on $\partial B_1$ is sometimes called the \emph{metric problem}.

The analysis of the metric problem in the case of first-order equations has a long history and goes back at least to Lions~\cite{Li}. The results in this section are related to some well-posedness results which appeared in~\cite{AS1}, although the treatment here is more general. The results in this section, in particular the continuity of $m_\mu^U$, are needed in the forthcoming papers on stochastic homogenization~\cite{AC,AT}, wherein they play an important technical role in ``localizing" the dependence of the maximal subsolutions $m_\mu$ on the random environment.

For $m_\mu^U$ to be well-defined, we require that the admissible set in its definition is nonempty. We introduce the critical parameter $\overline H_*(U)\in\R$ for which there exist subsolutions in $U$:
\begin{equation*}
\overline H_*(U) :=\inf \left\{ \mu\in\R \, : \, \exists  \ w\in \USC(U) \ \mbox{satisfying} \ -\tr\left( A(x) D^2u \right) + H(Du,x) \leq \mu  \ \mbox{in} \ U\right\}.
\end{equation*}

The main result of this section asserts that $m_\mu^U$ is locally Lipschitz continuous in $U\setminus \overline B_1$ and characterizes it as the maximal solution of the following boundary-value problem:
\begin{equation} \label{e.sc.bvp}
\left\{ \begin{aligned}
 & -\tr\left( A(x) D^2u \right) + H(Du,x) \leq \mu &  \mbox{in} & \  U,\\
 & -\tr\left( A(x) D^2u \right) + H(Du,x) = \mu &  \mbox{in} & \ \overline U \setminus \overline B_1,\\
& u \leq 0 & \mbox{in} & \ B_1.
\end{aligned} \right.
\end{equation}
Note the state-constrained boundary conditions on $\partial U$.

\begin{thm}
\label{t.sc}
Let~$U\subseteq \Rd$ satisfy~\eqref{e.sc.U}  and $\mu \geq\overline H_*(U)$. Then the function $m_\mu^U$ defined in~\eqref{e.scsol} belongs to $C^{0,1}_{\mathrm{loc}}( U\setminus \overline B_1)$ and satisfies~\eqref{e.sc.bvp}.
\end{thm}

The proof of Theorem~\ref{t.sc} in the superquadratic case ($\m>2$) is relatively easy due to Lemma~\ref{l.superquad}. We present this argument separately before considering the more interesting subquadratic case that $1<\m\leq 2$. 

\begin{proof}[{\bf Proof of Theorem~\ref{t.sc} in the case $\m > 2$}]
By Lemma~\ref{l.superquad}, the family of subsolutions of~\eqref{e.sc} belonging to $\USC(U)$ which are nonpositive on $\overline B_1$ is bounded in $C^{0,\beta}(U_\delta \cap B_R)$ for each $R,\delta > 0$ and $\beta:= (\m-2)/(\m-1)$. Since $U$ is connected, we deduce that $m_\mu^U$ is locally bounded in $U$ and $m_\mu^U \in C^{0,\beta}(U_\delta\cap B_R)$ for every $R,\delta > 0$. In particular, $m_\mu^U\in C^{0,\beta}_{\mathrm{loc}}(U)$. 

We have left to check that $m_\mu^U$ satisfies~\eqref{e.sc.bvp} by the usual Perron argument. It is clear from its definition as a supremum of subsolutions that $m_\mu^U$ is a subsolution of~\eqref{e.sc}. To argue that it is a supersolution, we select a smooth function $\phi \in C^{\infty}(\overline U\setminus \overline B_1)$ and a point $x_0\in  \overline U \setminus \overline B_1$ such that 
\begin{equation}\label{e.sc.supq.touch}
m_\mu^U - \phi \quad \mbox{has a strict local minimum at} \ x_0. 
\end{equation}
We must show that 
\begin{equation}\label{e.sc.supq.wts}
-\tr(A(x_0)D^2\phi(x_0)) + H(D\phi(x_0),x_0) \geq \mu. 
\end{equation}
Assuming on the contrary that~\eqref{e.sc.supq.wts} is false, we use the smoothness of $\phi$ to find $r,\theta> 0$ such that 
\begin{equation*}\label{}
-\tr(A(x_0)D^2\phi(x_0)) + H(D\phi(x_0),x_0) \leq \mu - \theta \quad \mbox{in} \ B_{r}(x_0)\cap \overline U.
\end{equation*}
By adding a constant to $\phi$ and shrinking $r>0$, if necessary, we may assume by~\eqref{e.sc.supq.touch} that 
\begin{equation*}\label{}
m^U_\mu(x_0) - \phi(x_0) < 0 < m^U_\mu - \phi \quad \mbox{on} \  \partial B_r(x_0)\cap \overline U.
\end{equation*}
Define the function
\begin{equation*}\label{}
w(x):= \begin{cases}
m^U_\mu(x) & \mbox{in} \ \overline U \setminus  \overline B_r(x_0), \\
\max\{ m^U_\mu(x) , \phi(x) \} & \mbox{in} \ \overline U \cap \overline B_r(x_0).
\end{cases}
\end{equation*}
It is clear from construction that $w$ is a subsolution of~\eqref{e.sc} and hence an admissible function in the definition of $m_\mu^U$. Thus $w\leq m_\mu^U$ in $U$. This contradicts the fact that $w = \phi > m_\mu^U$ in a neighborhood of $x_0$. 
\end{proof}

\begin{proof}[{\bf Proof of Theorem~\ref{t.sc} in the case $1<\m\leq 2$}]

We use the method introduced in the proof of Theorem~\ref{t.state.q}. We assume first that $U$ is bounded and smooth and $\mu > \overline H_*(U)$, and remove these assumptions in the last step. For convenience, we drop the dependence of $m_\mu^U$ on $U$, writing $m_\mu=m_\mu^U$, until the last step.

The strategy is to consider maximal subsolutions of the following perturbed equation
\begin{equation} \label{e.sc.ep}
-\tr\left( A(x) D^2w \right) + H(Dw,x) = \mu +\ep (\db_U)^{-\frac{m}{m-1}} \quad \mbox{in}  \  U\setminus \overline B_1.
\end{equation}
The nonnegative function $\db_U$ is as defined in Section~\ref{BVSC}. We denote by $m_\mu^\ep$ the corresponding maximal subsolution of~\eqref{e.sc.ep}, i.e.
\begin{equation} \label{max.sc.ep}
m_\mu^{\ep}(x):= \sup\left\{ w(x) \,:\, w\in \USC( U ) \ \mbox{is a subsolution of~\eqref{e.sc.ep} and} \ w \leq 0 \ \mbox{on} \ \overline B_1 \right\}.
\end{equation}
It is clear that $m_\mu^\ep$ is monotone in~$\ep$, since $\db_U$ is nonnegative:  for all $0<\ep <\ep'<1$, we have
\begin{equation} \label{mmuepmono}
m_\mu \le m_\mu^\ep \le m_\mu^{\ep'} \quad \mbox{in} \ U.
\end{equation}
We first argue that $\{ m_\mu^\ep\}_{0<\ep<1}$ is uniformly Lipschitz continuous in each $U_\delta\setminus \overline B_1$ and satisfies~\eqref{e.sc.ep} and then obtain the theorem after arguing that $m^\ep_\mu \to m_\mu$ as $\ep \to 0$. 

{\it Step 1.} We derive upper bounds for $m_\mu^1$ on $U_\delta$. Precisely, we claim that, for some $C>0$, 
\begin{equation} \label{e.big.gun1}
\sup_{U_\delta} m_\mu^1 \leq C. 
\end{equation}
Note that this gives uniform (in $0<\ep<1$) upper bounds for $m^\ep_\mu$ in view of~\eqref{mmuepmono}. To obtain this estimate it is necessary to use a covering argument, since the bound depends on the geometry of~$U$. We may select $0 < r< \delta/4$ and $K \in \mathbb N$, depending only on~$U$, such that $B_{1+4r} \subseteq U_\delta$ and, for each $z \in U_\delta$, there exist $n\leq K$ and $x_1,\ldots, x_n \in U_{4r}\setminus \overline B_1$ such that
\begin{equation*}
x_1 \in B_{1+r}, \quad x_n=z \quad \mbox{and} \quad |x_{i+1}-x_i|\leq 2r  \quad \text{for every}\ 1\le i \le n-1.
\end{equation*}
By exhibiting explicit, smooth supersolutions, we will show that 
\begin{equation} \label{e.big.gun2}
\sup_{B_{1+r}} m^1_\mu \leq C
\end{equation}
and, for each $i\geq 1$,
\begin{equation}\label{e.big.gun3}
\sup_{B(x_{i+1},r)} m^1_\mu \leq C + \sup_{B(x_{i},r)} m^1_\mu,
\end{equation}
for some $C> 0$ to be determined. This yields the desired estimate.

Here are the test functions: for $x \in B_{1+4r} \setminus B_1$, we set
\begin{equation*}\label{}
\psi_0(x):= \begin{cases} 
\left (\dist(x,\partial B_{1+4r})\right)^{-\frac{2-\m}{\m-1}} & 1< \m <2, \\
 1 - \log\left(  \dist(x,\partial B_{1+4r})\right)  & \m = 2,
\end{cases}
\end{equation*}
and, for each $1 \leq i \leq n-1$ and $x \in B(x_i,4r) \setminus B(x_i,r)$,
\begin{equation*}\label{}
\psi_i(x):= \begin{cases} 
\left (\dist(x,\partial B(x_i,4r))\right)^{-\frac{2-\m}{\m-1}} &   1< \m <2, \\
 1 - \log\left(  \dist(x,\partial B(x_i,4r)\right)  & \m = 2.
\end{cases}
\end{equation*}
Notice that $\psi_i$ is smooth for all $0\le i \le n-1$. Set
\begin{equation*}
\tilde \mu:=\mu+\sup_{U_\delta} (\db_U)^{-\frac{m}{m-1}}.
\end{equation*}
A routine computation confirms that for~$C>0$ sufficiently large (but independent of~$i$,~$x_i$ and~$r$), the function $w:=C\psi_0$ satisfies
\begin{equation}\label{C.large1}
-\tr \left (A(x) D^2w \right) + H\left (Dw,x\right) >\tilde \mu \quad \text{in} \ B_{1+4r} \setminus \overline B_1
\end{equation}
and, for each $1\leq i \leq n-1$, the function $w:=C\psi_i$ satisfies
\begin{equation}\label{C.large2}
-\tr \left (A(x) D^2w \right) + H\left (Dw,x\right) >\tilde \mu \quad \text{in} \ B(x_i,4r) \setminus \overline B(x_i,r).
\end{equation}
Since $\psi_0(x) \to +\infty$ and $\psi_i(x)\to +\infty$ as $x\to \partial B_{1+4r}$ and $x\to \partial B(x_i,4r)$, respectively, we deduce that the functions $m_\mu^\ep - C\psi_0$ and  $m_\mu^\ep - C\psi_i$ do not possess local maximums in $B_{1+4r} \setminus \overline B_1$ and $B(x_i,4r) \setminus \overline B(x_i,r)$, respectively. It follows that
\begin{equation} \label{e.mulocbns}
m^1_\mu \leq C\psi_0 - \min_{ \partial\overline B_1} \big(C\psi_0\big) \quad \text{in} \ B_{1+4r}\setminus \overline B_1
\end{equation}
and
\begin{equation*}
m^1_\mu \leq C\psi_i+ \max_{\overline B(x_i,r)} \big(m_\mu^1- C\psi_i\big) \quad \text{in} \ B(x_i,4r) \setminus \overline B(x_i,r).
\end{equation*}
In view of the fact that $\overline B(x_{2},r) \subseteq \overline B_{1+3r} \subseteq B_{1+4r}$ and $\overline B(x_{i+1},r) \subseteq \overline B(x_i,3r) \subseteq B(x_i,4r)$, we obtain the estimates~\eqref{e.big.gun2} and~\eqref{e.big.gun3}. This completes the proof of~\eqref{e.big.gun1}. 

{\it Step 2.} We show that there exist constants $c,C>0$, which may  depend on $\ep$, such that
\begin{equation}\label{c-C-ep-m}
c \zeta_U-C \leq m_\mu^\ep \leq C \zeta_U \quad \text{in}\ U \setminus \overline U_\delta.
\end{equation}
Here $\zeta_U$ is defined as in~\eqref{e.sc.phiU}. By a direct computation similar to the one in the previous section, we can find positive constants $\delta,c,C>0$ such that $c\zeta_U$ is a smooth, strict subsolution of \eqref{e.sc.ep} in $U \setminus \overline U_\delta$ (here is where we need the help of the term $\ep \db_U$ on the right-hand side) and, for every $0<s<\delta$, the function $C\zeta_{U_s}$ is a smooth, strict supersolution of \eqref{e.sc.ep} in $U_s \setminus \overline U_\delta$. By step 1 above, we can pick~$\delta > 0$ sufficiently small and~$C>0$ sufficiently large so that $m_\mu^\ep <C \zeta_{U_s}$ on $\partial U_\delta$.  Since $\zeta_{U_s}$ is smooth, the definition of viscosity subsolution yields that $m_\mu^\ep \le C\zeta_{U_s}$ in $U_s \setminus \overline U_\delta$, since this must be true for any function in the admissible class in the definition of $m_\mu^\ep$. Sending $s\to 0$ implies the second inequality in  \eqref{c-C-ep-m}. To get the first inequality, we use the fact that $\mu > \overline H_*$ to select $\nu < \mu$ and a subsolution $v\in \USC(U)$ of
\begin{equation*} \label{}
-\tr\left( A(x) D^2v \right) + H(Dv,x) = \nu  \quad \mbox{in}  \  U.
\end{equation*}
By subtracting a constant, we may assume that $\sup_{B_1} v = 0$. Now consider the function
\begin{equation*}
\widetilde v(x) :=
\max \big\{v(x), c\zeta_U (x)- \max_{\overline U_{\delta/2}} \left( c\zeta_U -v\right) \big\}.
\end{equation*}
Observe $\widetilde v$ is equal to $v$ in $U_{\delta/2}$ and hence a subsolution of \eqref{e.sc.ep} in $U$. By the definition of~$m_\mu^\ep$, we deduce that $m_\mu^\ep \geq \widetilde v$, which gives the first inequality of~\eqref{c-C-ep-m} in view of the fact that $\max_{\overline U_{\delta/2}} (c\zeta_U - v) < +\infty$.

In the next three steps we show that $m_\mu^\ep=\widehat m_\mu^\ep$, where we introduce $\widehat m_\mu^\ep$ as the minimal supersolution of~\eqref{e.sc.ep}, defined by
\begin{multline*}
\widehat m_\mu^\ep(x):=\inf \Big\{w(x)\, : \, w \in \LSC(U \setminus B_1)  \ \text{in a supersolution of \eqref{e.sc.ep} on $U\setminus B_1$, and } \\ w(x) \rightarrow +\infty \ \mbox{as} \ x \to \partial U \Big\}.
\end{multline*}

{\it Step 3.} Estimates on the blow-up rate of $\widehat m_\mu^\ep$ near $\partial U$: we show that there exists $c>0$ such that
\begin{equation}\label{c-hat-ep-m}
\widehat m_\mu^\ep \geq c\zeta_U -C \quad \text{in} \ U \setminus \overline U_\delta.
\end{equation}
As in Step 2, we have $c\zeta_{U^s}$ is a smooth, strict subsolution of  \eqref{e.sc.ep} in $U \setminus \overline U_\delta$ provided that  $c,s>0$ are chosen sufficiently small. Let $w \in \LSC(U\setminus B_1)$ be any function in the admissible class in the definition of $\widehat m_\mu^\ep$. 
Let $\widetilde v_s$ be defined in the same way as $\widetilde v$, but with $U^s$ in place of $U$. Then $\widetilde v_s$ is a strict subsolution of~\eqref{e.sc.ep}, and using the fact that $w(x)\to +\infty$ as $x\to \partial U$,  the comparison principle yields that $w \geq \widetilde v_s$. Sending $s\to 0$ yields that $w \geq \tilde v$. This completes the proof of~\eqref{c-hat-ep-m}.

{\it Step 4.} We show that $(m_\mu^\ep)^*$ and $(m_\mu^\ep)_*$ are a subsolution and a supersolution of \eqref{e.sc.ep}, respectively. This is by the standard Perron argument which is nearly the same as in the proof of the Theorem in the case $m>2$, above. Therefore we omit the argument. 

{\it Step 5.} We show finally that $m_\mu^\ep=\widehat m_\mu^\ep$. According to Lemma~\ref{l.convex2}, for every $\alpha > 0$, the function $w:= (1+\alpha) \widehat m_\mu^\ep - \alpha m_\mu^\ep$ is a supersolution of~\eqref{e.sc.ep} in $U \setminus \overline B_1$. Moreover, by  \eqref{c-C-ep-m} and~\eqref{c-hat-ep-m}, for sufficiently small $\alpha> 0$, the function $w$ satisfies $w \geq \frac12 c \zeta_U$ near $\partial U$. In particular, $w(x) \to +\infty$ as $x\to \partial U$. Therefore, we conclude from the definition of $\widehat m_\mu^\ep$ that $\widehat m_\mu^\ep \leq w$. A rearrangement of this inequality gives $m_\mu^\ep \leq \widehat m_\mu^\ep$. In view of the fact that, by Step 4, we have $m_\mu^\ep \geq (m_\mu^\ep)_* \geq \widehat m_\mu^\ep$, we deduce that $(m_\mu^\ep)^*=m_\mu^\ep = \widehat m_\mu^\ep = (m_\mu^\ep)_* $ and in particular, $m_\mu^\ep$ is continuous in $U \setminus \overline B_1$ and is a solution of~\eqref{e.sc.ep} in $U \setminus \overline B_1$.

{\it Step 6.} We complete the proof of the theorem in the case that $U$ is bounded. In view of Step 5 and Theorem~\ref{t.LIP}, $m^\ep_\mu\in C^{0,1}_{\mathrm{loc}}(U\setminus \overline B_1)$. Define
\begin{equation*} \label{}
m_\mu^0(x):= \limsup_{\ep\to0, \ y\to x} m^\ep_\mu(y).
\end{equation*}
As a half-relaxed limit, $m_\mu^0$ is a subsolution of~\eqref{e.sc} in $U$ and clearly $m_\mu^0 \leq 0$ in $\overline B_1$. Hence $m_\mu^0 \leq m_\mu$ in $U$ by the definition of the latter. By~\eqref{mmuepmono}, $m_\mu^0$ is the locally uniform pointwise limit of $m_\mu^\ep$ in $U\setminus B_1$, hence $m^0_\mu \in C^{0,1}_{\mathrm{loc}}(U\setminus \overline B_1)$, and we have $m^0_\mu \geq m_\mu$ in $U\setminus \overline B_1$. Thus $m^0_\mu \equiv m_\mu$ in $U\setminus \overline B_1$ and in particular $m_\mu$ belongs to $C^{0,1}_{\mathrm{loc}}(U\setminus \overline B_1)$. The stability of viscosity solutions under uniform limits yields that $m_\mu$ is a supersolution of~\eqref{e.sc} in $U$. 

\emph{Step 7.} We remove the assumption that $U$ is bounded and smooth and that $\mu > \overline H_*(U)$. It is immediate from the definition of $m^U_\mu$ that, for all domains $U, V\subseteq \Rd$ satisfying~\eqref{e.scsol}, we have 
\begin{equation}\label{e.sc.mono}
V \subseteq U \quad \mbox{implies that} \qquad m_\mu^U \leq m^V_\mu \quad \mbox{in} \ V.
\end{equation}
Therefore, for general $U$, we simply take an increasing sequence of bounded, smooth domains $V_1 \subset V_2 \subset\ldots$ such that $U = \cup_{k\in\N} V_k$ and deduce, in view of the argument above and the fact that the functions $m_\mu^{V_k}$ are locally Lipschitz in each $V_{j}$ uniformly in $k>j$, that $m_\mu^U$ is the local uniform limit of $m_\mu^{V_k}$. To obtain the result for $\mu=\overline H_*(U)$, we argue similarly, using the monotonicity of the map $\mu \mapsto m_\mu$ and the fact that $\sup_{B_1} m_\mu^U = 0$. 
\end{proof}

An important property of $m_\mu = m_\mu^{\Rd}$ is its \emph{subadditivity}, which is summarized in the following lemma. To state it, we let $m_\mu(\cdot,z)$ denote the analogue of $m_\mu$ with $B_1(z)$ in place of $B_1$, that is, for every $\mu > \overline H_*(\Rd)$
\begin{equation*} \label{}
m_\mu(y,z):= \sup\left\{ w(y) \,:\, w\in \USC( \Rd ) \ \mbox{is a subsolution of~\eqref{e.sc} and} \ w \leq 0 \ \mbox{on} \ \overline B_1(z) \right\}.
\end{equation*}
We also denote
\begin{equation*} \label{}
\widetilde m_\mu(y,z):= \sup_{B_1(y)} m_\mu(\cdot,z). 
\end{equation*}

\begin{lem}
\label{l.subadd}
For every $\mu > \overline H_*(\Rd)$ and $x,y,z\in\Rd$,
\begin{equation} \label{e.subadd}
\widetilde m_\mu(y,z) \leq \widetilde m_\mu(y,x) + \widetilde m_\mu(x,z).
\end{equation}
\end{lem}
\begin{proof}
Observe that $m_\mu(\cdot,z)- \sup_{B_1(x)} m_\mu(\cdot,z)$ is a subsolution of~\eqref{e.sc} in $\Rd$ which is nonpositive on $B_1(x)$. It therefore follows from the definition of $m_\mu(\cdot,x)$ that, for all $\xi\in \Rd$,
\begin{equation*} \label{}
m_\mu(\xi,z) - \widetilde m_\mu(x,z) = m_\mu(\xi,z)- \sup_{B_1(x)} m_\mu(\cdot,z) \leq m_\mu(\xi,x).
\end{equation*}
Taking the supremum over $\xi\in B_1(y)$ yields the lemma. 
\end{proof}

\begin{remark}{(Stochastic optimal control interpretation of $m_\mu$)} 
\label{r.control}
We think of $m_\mu(y,z)$ as measuring the ``cost" imposed by the environment for moving a particle from $y$ to $\overline B_1(z)$. We briefly summarize how this is made rigorous. We may write
\begin{equation*} \label{}
m_\mu(y,z,\omega) = \inf_{\theta_z, \,\alpha(\cdot)}  E_{\alpha,y} \left[ \int_{0}^{\theta_z} \left( \mu + L(-\alpha_s,X_s,\omega) \right)\, ds \, : \, \theta_z < \infty \right]
\end{equation*}
where $L$ is the Legendre-Fenchel transform of $H$, $\alpha_s$ is an $\Rd$-valued adapted process, $\theta_z$ is a stopping time, and with respect to a probability measure $P_{\alpha,y}$ on the space of paths (with expectation denoted by $E_{\alpha,y}$), the process $X_s$ solves the following SDE:
\begin{equation*} \label{}
dX_s = \alpha_s \, ds + \sigma(X_s,\omega) \, dB_s,
\end{equation*}
where $\sigma:=  (2 A)^{\frac12}$, $B_s$ is Brownian motion with respect to $P_{\alpha,y}$, and the control $\theta_z$ is an adapted stopping time for which $X_{\theta_z} \in B_1(z)$. The interpretation is that the controller can choose (or not) to stop if the diffusion is in $B_1(z)$. The proof is a straightforward exercise involving the dynamic programming principle; as we don't use the stochastic control interpretation in our arguments, we omit the argument.
\end{remark}

\begin{remark}
It is immediate from the convexity of $H$ and Lemma~\ref{l.convex} that the map $\mu \mapsto m_\mu$ is concave.
\end{remark}

\begin{remark}
The Lipschitz estimate~\eqref{e.Lmu} yields, for every $x\in \Rd\setminus B_3$, the bound 
\begin{equation} \label{e.oscbnd}
\osc_{B_1(x)} m_\mu \leq C \left[ \left( \frac{(1+\Lambda_1)^{1/2} \Lambda_2 }{a_2(x)} \right)^{2/(\m-1)} + \left( \frac{M_2(x)+\mu}{a_2(x)} \right)^{1/\m}  \right],
\end{equation}
where $a_2(x)$ and $M_2(x)$ are constants as in~\eqref{e.Hsubq} and~\eqref{e.HsubqLip} for the shifted  Hamiltonian $H(\cdot,\cdot-x)$. We may extend~\eqref{e.oscbnd} to an oscillation bound for all $x$ as follows. First, we have the upper bound $\sup_{x\in B_4} m_\mu(x,0) \leq C$ by Step~2 of the proof of Theorem 5.1. Next, as in the proof of Lemma~\ref{l.subadd}, we have, for every $x\in B_4$,
\begin{equation*} \label{}
m_\mu(x,0) \geq m_\mu(x,z) - \sup_{B_1} m_\mu(\cdot,z) \geq - \osc_{B_4} m_\mu(\cdot,z). 
\end{equation*}
Taking $|z| >7$ and combining this with~\eqref{e.oscbnd}, we obtain that, for every $x\in \Rd$,
\begin{equation} \label{e.oscbnd2}
\osc_{B_1(x)} m_\mu \leq C \left[ \left( \frac{(1+\Lambda_1)^{1/2} \Lambda_2 }{a_5(x)} \right)^{2/(\m-1)} + \left( \frac{M_5(x)+\mu}{a_5(x)} \right)^{1/\m}  \right].
\end{equation}
\end{remark}

\begin{remark}
Theorem~\ref{t.sc} may be extended to other ``target sets." Fix a compact set $K \subseteq \Rd$, assume that $U \subseteq \Rd$ is a domain satisfying
\begin{equation*} \label{}
K+\overline B_1 \subseteq U, \quad U \ \ \mbox{and} \ \ U\setminus (K+\overline B_1) \ \  \mbox{are connected}
\end{equation*}
and define, for every $\mu \geq \overline H_*(U)$,
\begin{equation*} \label{}
m^U_\mu(x,K):= \sup\left\{ w(x) \,:\, w\in \USC( U ) \ \mbox{is a subsolution of~\eqref{e.sc} and} \ w \leq 0 \ \mbox{on} \ K+\overline B_1 \right\}.
\end{equation*}
The argument in the proof of Theorem~\ref{t.sc} yields that $m_\mu^U(\cdot,K) \in C^{0,1}_{\mathrm{loc}}(U\setminus (K+\overline B_1))$ is the maximal subsolution $u$ of~\eqref{e.sc} subject to $u \leq 0$ on $K+\overline B_1$. The main difference in the argument comes in the proof of the bound~\eqref{e.big.gun1}, in which one needs to compare to the test function $\phi_0(\cdot-y)$ for every point $y\in K$ to get the analogue of~\eqref{e.big.gun2}. This adaptation is straightforward and left to the reader.
\end{remark}

\subsection*{Acknowledgements}
S. Armstrong thanks the Forschungsinstitut f\"ur Mathematik (FIM) of ETH Z\"urich for support.

\bibliographystyle{plain}
\bibliography{viscoustheory}

\end{document}